\documentclass[submission,copyright,creativecommons]{eptcs}
\pdfoutput=1
\usepackage{breakurl}             %
\usepackage{underscore}           %

\usepackage{macros}

\newcommand{\Run}{\colorPre{\operatorname{Run}}}

\usepackage[leftmargin=1em,rightmargin=1ex,vskip=1ex]{quoting}
\usepackage{enumitem}
\usepackage{lscape}

\title{Promonads and String Diagrams for Effectful Categories}
\author{Mario Rom\'an
  \institute{Tallinn University of Technology}
  \email{mroman@ttu.ee}
}

\makeatletter
\def\@copyrightspace{\relax}
\makeatother

\begin{document}
\maketitle

\begin{abstract}
  Premonoidal and Freyd categories are both generalized by non-cartesian Freyd categories: effectful categories.
  We construct string diagrams for effectful categories in terms of the string diagrams for a monoidal category with a freely added object.
  We show that effectful categories are pseudomonoids in a monoidal bicategory of promonads with a suitable tensor product.
\end{abstract}

\section{Introduction}

Category theory has two sucessful applications that are rarely combined: monoidal string diagrams \cite{joyal91} and functional programming semantics \cite{moggi91}.
We use string diagrams to talk about quantum transformations \cite{abramsky2009categorical}, relational queries \cite{bonchi18}, and even computability \cite{pavlovic13}; at the same time, proof nets and the geometry of interaction \cite{girard89,blute96} have been widely applied in computer science \cite{abramsky02,hoshino14}.
On the other hand, we traditionally use monads and comonads, Kleisli categories and \premonoidalCategories{} to explain effectful functional programming \cite{hughes00,jacobs09,moggi91,power99:freyd,uustalu2008comonadic}. Even if we traditionally employ Freyd categories with a cartesian base \cite{power02}, we can also consider non-cartesian Freyd categories \cite{staton13}, which we call \emph{effectful categories}.

\vspace{-1em}
\paragraph{Contributions.} These applications are well-known. However, some foundational results in the intersection between string diagrams, \premonoidalCategories{} and \effectfulCategories{} are missing in the literature. This manuscript contributes two such results.
\begin{itemize}
  \item
    We introduce string diagrams for effectful categories.
    Jeffrey \cite{jeffrey97} was the first to preformally employ string diagrams of \premonoidalCategories{}.
    His technique consists in introducing an extra wire -- which we call the \emph{runtime} -- that prevents some morphisms from interchanging.
    We promote this preformal technique into a result about the construction of free premonoidal, \Freyd{} and \effectfulCategories{}:
    the free premonoidal category can be constructed in terms of the free monoidal category with an extra wire.

    Our slogan, which constitutes the statement of \Cref{theorem:runtime-as-a-resource}, is
    \begin{quote}
      \emph{``Premonoidal categories are Monoidal categories with a Runtime.''}
    \end{quote}

  \item
    We prove that \effectfulCategories{} are \promonad{} \pseudomonoids{}.
    Promonads are the profunctorial counterpart of monads; they are used to encode effects in functional programming (where they are given extra properties and called \emph{arrows} \cite{hughes00}).
    We claim that, in the same way that monoidal categories are pseudomonoids in the bicategory of categories \cite{street97}, premonoidal \effectfulCategories{} are pseudomonoids in a monoidal bicategory of promonads.
    This result justifies the role of \effectfulCategories{} as a foundational object.
\end{itemize}

\subsection{Synopsis}
\Cref{sec:premonoidalcats,sec:effectfulcats} contain mostly preliminary material on premonoidal, Freyd and effectful categories.
Our first original contribution is in \Cref{sec:runtime}; we prove that premonoidal categories are monoidal categories with runtime (\Cref{theorem:runtime-as-a-resource}).
\Cref{sec:promonads} makes explicit the well-known theory of profunctors, promonads and identity-on-objects functors.
In \Cref{sec:puretensor}, we introduce the pure tensor of promonads. We use it in \Cref{sec:pseudomonoids} to prove our second main contribution (\Cref{th:freydpseudomonoid}).

\section{Premonoidal and Effectful Categories}
\label{sec:premonoidal}

\subsection{Premonoidal categories}
\label{sec:premonoidalcats}
\PremonoidalCategories{} are monoidal categories without the \emph{interchange law},
$(\fm \tensor \im) \comp (\im \tensor \gm) \neq (\im \tensor \gm) \comp (\fm \tensor \im)$.
This means that we cannot tensor any two arbitrary morphisms, $(\fm \tensor \gm)$, without explicitly stating which one is to be composed first, $(\fm \tensor \im) \comp (\im \tensor \gm)$ or $(\im \tensor \gm) \comp (\fm \tensor \im)$, and the two compositions are not equivalent (\Cref{fig:noninterchange}).
  \begin{figure}[H]
    \centering
    \includegraphics[scale=0.6]{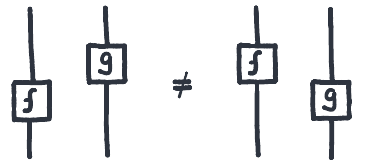}
    \caption{The interchange law does not hold in a premonoidal category.}
    \label{fig:noninterchange}
  \end{figure}

In technical terms, the tensor of a \premonoidalCategory{} $(\tensor) \colon \catC \times \catC \to \catC$ is not a functor, but only what is called a \emph{\sesquifunctor{}}: independently functorial on each variable. Tensoring with any identity is itself a functor $(\bullet \tensor \im) \colon \catC \to \catC$, but there is no functor $(\bullet \tensor \bullet) \colon \catC \times \catC \to \catC$.

A good motivation for dropping the interchange law can be found when describing transformations that affect some global state.
These effectful processes should not interchange in general, because the order in which we modify the global state is meaningful.
For instance, in the Kleisli category of the \emph{writer monad}, $(\Sigma^{\ast} \times \bullet) \colon \Set \to \Set$ for some alphabet $\Sigma \in \Set$, we can consider the function $\mathsf{print} \colon \Sigma^{\ast} \to \Sigma^{\ast} \times 1$. The order in which we ``print'' does matter (\Cref{fig:writermonad}).
  \begin{figure}[H]
    \centering
    \includegraphics[scale=0.6]{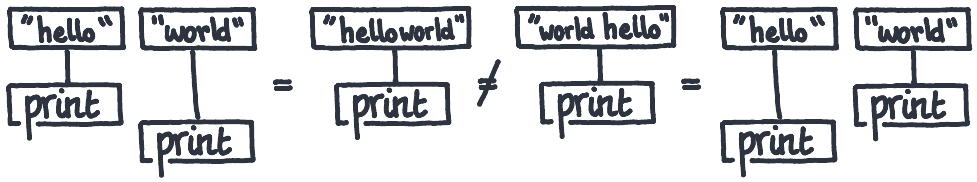}
    \caption{Writing does not interchange.}
    \label{fig:writermonad}
  \end{figure}

  Not surprisingly, the paradigmatic examples of premonoidal categories are the Kleisli categories of Set-based monads $T \colon \Set \to \Set$  (more generally, of strong monads), which fail to be monoidal unless the monad itself is commutative \cite{guitart1980tenseurs,power97,power99:freyd,hedgesblog2019}.
  Intuitively, the morphisms are ``effectful'', and these effects do not always commute.

However, we may still want to allow some morphisms to interchange.
For instance, apart from asking the same associators and unitors of monoidal categories to exist, we ask them to be \emph{central}: that means that they interchange with any other morphism.
This notion of centrality forces us to write the definition of premonoidal category in two different steps: first, we introduce the minimal setting in which centrality can be considered (\emph{binoidal} categories \cite{power99:freyd}) and then we use that setting to bootstrap the full definition of premonoidal category with central coherence morphisms.

\begin{definition}[Binoidal category]
  \defining{linknoidalcategory}{}
  A \emph{binoidal category} is a category $\catC$ endowed with an object $I \in \catC$ and an object $A \tensor B$ for each $A \in \catC$ and $B \in \catC$.
  There are functors
  $(A \tensor \bullet) \colon \catC \to \catC
    \mbox{, and }
    (\bullet \tensor B) \colon \catC \to \catC$
  that coincide on $(A \tensor B)$, even if $(\bullet \tensor \bullet)$ is not itself a functor.
\end{definition}

Again, this means that we can tensor with identities (whiskering), functorially; but we cannot tensor two arbitrary morphisms: the interchange law stops being true in general.
The \emph{centre}, $\zentre(\catC)$, is the wide subcategory of morphisms that do satisfy the interchange law with any other morphism.
That is, $f \colon A \to B$ is \emph{central} if, for each $g \colon A' \to B'$,
\begin{align*}
  (f \tensor \im_{A'}) \comp (\im_{B} \tensor g)
   = (\im_{A} \tensor g) \comp (f \tensor \im_{B'}), \mbox{ and }
  (\im_{A'} \tensor f) \comp (g \tensor \im_{B})
   = (g \tensor \im_{A}) \comp (\im_{B'} \tensor f).
\end{align*}

\begin{definition}
  \defining{linkpremonoidalcategory}
  A \emph{premonoidal category} is a \noidalCategory{} $(\catC,\tensor,I)$ together with the following coherence isomorphisms
  $\alpha_{A,B,C} \colon A \tensor (B \tensor C) \to (A \tensor B) \tensor C$, $\rho_{A} \colon A \tensor I \to A$ and $\lambda_{A} \colon I \tensor A \to A$ which are central, natural \emph{separately at each given component}, and satisfy the pentagon and triangle equations.

  A \premonoidalCategory{} is \emph{strict} when these coherence morphisms are identities.
  A \premonoidalCategory{} is moreover \emph{symmetric} when it is endowed with a coherence isomorphism $\sigma_{A,B} \colon A \tensor B \to B \tensor A$ that is central and natural at each given component, and satisfies the symmetry condition and hexagon equations.
\end{definition}

\begin{remark}
  The coherence theorem of monoidal categories still holds for premonoidal categories: every premonoidal is equivalent to a strict one.
  We will construct the free strict premonoidal category using string diagrams.
  However, the usual string diagrams for monoidal categories need to be restricted: in premonoidal categories, we cannot consider two morphisms in parallel unless any of the two is \emph{central}.
\end{remark}

\subsection{Effectful and Freyd categories}
\label{sec:effectfulcats}

\PremonoidalCategories{} immediately present a problem: what are the strong premonoidal functors?
If we want them to compose, they should preserve centrality of the coherence morphisms (so that the central coherence morphisms of $F \comp G$ are these of $F$ after applying $G$), but naively asking them to preserve all central morphisms rules out important examples~\cite{staton13}.
The solution is to explicitly choose some central morphisms that represent ``pure'' computations.
These do not need to form the whole centre: it could be that some morphisms considered  \emph{effectful} just ``happen'' to fall in the centre of the category, while we do not ask our functors to preserve them.
This is the well-studied notion of a \emph{non-cartesian Freyd category}, which we shorten to \emph{effectful monoidal category} or \emph{effectful category}.\footnote{The name ``Freyd category'' sometimes assumes cartesianity of the pure morphisms, but it is also used for the general case.
Choosing to call ``effectful categories'' to the general case and reserving the name ``Freyd categories'' for the cartesian ones avoids this clash of nomenclature.
There exists also the more fine-grained notion of ``Cartesian effect category'' \cite{dumas11}, which generalizes Freyd categories and may further justify calling ``effectful category'' to the general case.}

\EffectfulCategories{} are \premonoidalCategories{} endowed with a chosen family of central morphisms.
These central morphisms are called \pure{} morphisms, constrasting with the general, non-central, morphisms that fall outside this family, which we call \effectful{}.

\begin{definition}
  \defining{linkeffectful}{}
  \defining{linkfreyd}{}
  An \emph{effectful category}
  is an identity-on-objects functor $\colorMon{\baseV} \to \colorPre{\catC}$ from a monoidal category $\baseV$ (the \pure{} morphisms, or ``values'') to a premonoidal category $\catC$ (the \effectful{} morphisms, or ``computations''), that strictly preserves all of the premonoidal structure and whose image is central. It is \emph{strict} when both are. A \emph{Freyd category} \cite{levy2004} is an effectful category where the \pure{} morphisms form a cartesian monoidal category.
\end{definition}

Effectful categories solve the problem of defining premonoidal functors: a functor between effectful categories needs to preserve only the \pure{} morphisms.
We are not losing expressivity: premonoidal categories are effectful with their centre, $\colorMon{\zentre}(\catC) \to \catC$.
From now on, we study \effectfulCategories{}.

\begin{definition}[Effectful functor]
  \defining{linkeffectfulfunctor}
  Let $\cbaseV \to \ccatC$ and $\cbaseW \to \ccatD$ be effectful categories.
  An \emph{effectful functor} is a quadruple  $(F,F_{0},\varepsilon, \mu)$ consisting of a functor $F \colon \ccatC \to \ccatD$ and a functor $F_{0} \colon \cbaseV \to \cbaseW$ making the square commute,
  and two natural and pure isomorphisms $\varepsilon \colon J \cong F(I)$ and $\mu \colon F(A \otimes B) \cong F(A) \otimes F(B)$ such that they make $F_{0}$ a monoidal functor. It is \emph{strict} if these are identities.
\end{definition}

When drawing string diagrams in an effectful category, we shall use two different colours to declare if we are depicting either a value or a computation (\Cref{fig:prehelloworld}).

  \begin{figure}[H]
    \centering
    \includegraphics[scale=0.6]{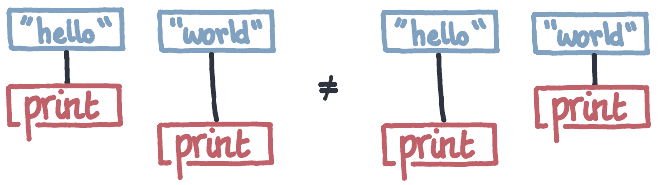}
    \caption{``Hello world'' is not ``world hello''.}
    \label{fig:prehelloworld}
  \end{figure}

Here, the values \colorMon{``hello''} and \colorMon{``world''} satisfy the interchange law as in an ordinary monoidal category. However, the effectful computation \colorPre{``print''} does not need to satisfy the interchange law.
String diagrams like these can be found in the work of Alan Jeffrey \cite{jeffrey97}.
Jeffrey presents a clever mechanism to graphically depict the failure of interchange: all effectful morphisms need to have a control wire as an input and output.
This control wire needs to be passed around to all the computations in order, and it prevents them from interchanging.

  \begin{figure}[H]
    \centering
    \includegraphics[scale=0.6]{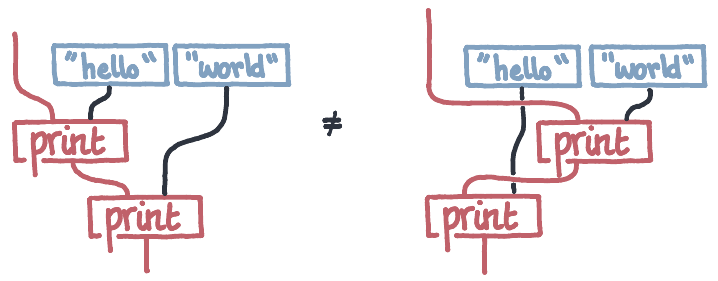}
    \caption{An extra wire prevents interchange.}
    \label{fig:helloworld}
  \end{figure}

A common interpretation of monoidal categories is as theories of resources.
We can interpret premonoidal categories as monoidal categories with an extra resource -- the \colorPre{``runtime''} -- that needs to be passed to all computations.
The next section promotes Jeffrey's observation into a theorem.

\subsection{Premonoidals are monoidals with runtime}
\label{sec:runtime}

String diagrams rely on the fact that the morphisms of the monoidal category freely generated over a \polygraph{} of generators are string diagrams on these generators, quotiented by topological deformations \cite{joyal91}.
We justify string diagrams for premonoidal categories by proving that the freely generated effectful category over a pair of polygraphs (for pure and effectful generators, respectively) can be constructed as the freely generated monoidal category over a particular polygraph that includes an extra wire.

\begin{definition}
  \defining{linkpolygraph}{}

  A \emph{\polygraph{}}  $\hyG$ (analogue of a \emph{multigraph} \cite{shulman2016categorical}) is given by a set of objects, $\hyGobj$, and a set of arrows $\hyG(A_{0},\dots,A_{n};B_{0},\dots,B_{m})$ for any two sequences of objects $A_{0},\dots,A_{n}$ and $B_{0},\dots,B_{m}$.
  A morphism of \polygraphs{} $f \colon \hyG \to \hyH$ is a function between their object sets, $f_{\mathrm{obj}} \colon \hyGobj \to \hyHobj$,
  and a function between their corresponding morphism sets,
  \[f_{A_{0},\dots,A_{n}; B_{0},\dots,B_{n}} \colon \hyG(A_{0},\dots,A_{n};B_{0},\dots,B_{m}) \to \hyH(f_{\mathrm{obj}}(A_{0}),\dots,f_{\mathrm{obj}}(A_{n});f_{\mathrm{obj}}(B_{0}),\dots,f_{\mathrm{obj}}(B_{m})).\]
    \defining{linkpolygraphcouple}{}A \emph{\polygraphCouple{}} is a pair of \polygraphs{} $(\hyV,\hyG)$ sharing the same objects, $\hyV_{\mathrm{obj}} = \hyG_{\mathrm{obj}}$.
  A morphism of \polygraphCouples{} $(u,f) \colon (\hyV,\hyG) \to (\hyW,\hyH)$ is a pair of morphisms of \polygraphs{}, $u \colon \hyV \to \hyW$ and $f \colon \hyG \to \hyH$, such that they coincide on objects, $f_{\mathrm{obj}} = u_{\mathrm{obj}}$.
\end{definition}

\begin{remark}
  There exists an adjunction between \polygraphs{} and strict monoidal categories.
  Any monoidal category $\catC$ can be seen as a \polygraph{} $\mathcal{U}_\catC$ where the edges $\mathcal{U}_{\catC}(A_{0},\dots, A_{n};B_{0}, \dots, B_{m})$ are the morphisms $\catC(A_{0} \tensor \dots \tensor A_{n},B_{0} \tensor \dots \tensor B_{m})$, and we forget about composition and tensoring.
  Given a \polygraph{} $\hyG$, the free strict monoidal category $\MON(\hyG)$ is the strict monoidal category that has as morphisms the string diagrams over the generators of the \polygraph{}.

  We will construct a similar adjunction between \polygraphCouples{} and \effectfulCategories{}. Let us start by formally adding the runtime to a free monoidal category.
\end{remark}

\begin{definition}[Runtime monoidal category]
  \defining{linkruntimepolygraph}{}
  Let $(\hyV,\hyG)$ be a \polygraphCouple{}.
  Its \emph{runtime monoidal category}, $\MONRUN(\hyV,\hyG)$, is the monoidal category freely generated from adding an extra object -- the runtime, $\R$ -- to the input and output of every effectful generator in $\hyG$ (but not to those in $\hyV$), and letting that extra object be braided with respect to every other object of the category.

  In other words, it is the monoidal category freely generated by the following \polygraph{}, $\Run(\hyV,\hyG)$,
  (\Cref{fig:rungen}),
  assuming $A_{0},\dots,A_{n}$ and $B_{0},\dots,B_{m}$ are distinct from $\R$
  \begin{itemize}
    \item $\obj{\Run(\hyV,\hyG)} = \obj{\hyG} + \{ \R \} = \obj{\hyV} + \{ \R \}$,
    \item $\Run(\hyV,\hyG)(\R,A_{0},\dots,A_{n};\R, B_{0},\dots,B_{n}) = \hyG(A_{0},\dots,A_{n}; B_{0},\dots,B_{n})$,
    \item $\Run(\hyV,\hyG)(A_{0},\dots,A_{n}; B_{0},\dots,B_{n}) = \hyV(A_{0},\dots,A_{n}; B_{0},\dots,B_{n})$,
    \item $\Run(\hyV,\hyG)(\R,A_{0};A_{0},\R) = \Run(\hyV,\hyG)(A_{0},\R;\R,A_{0}) = \{\sigma\}$,
  \end{itemize}
  with $\Run(\hyV,\hyG)$ empty in any other case, and quotiented by the braiding axioms for $\R$ (\Cref{fig:runaxiom}).
  \begin{figure}[H]
    \centering
    \includegraphics[scale=0.6]{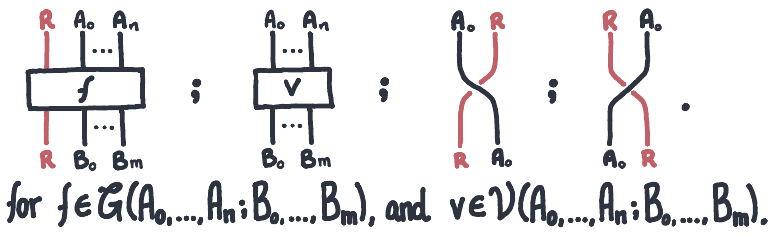}
    \caption{Generators for the runtime monoidal category.}
    \label{fig:rungen}
  \end{figure}
  \begin{figure}[H]
    \centering
    \includegraphics[scale=0.6]{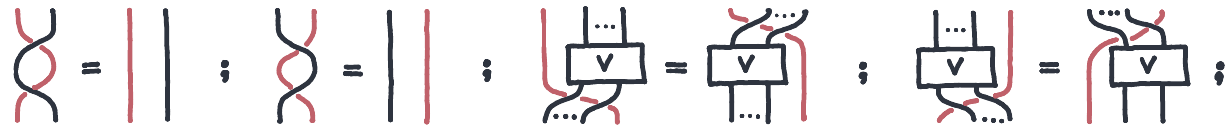}
    \caption{Axioms for the runtime monoidal category.}
    \label{fig:runaxiom}
  \end{figure}
\end{definition}

Somehow, we are asking the runtime $\R$ to be in the Drinfeld centre \cite{drinfeld10} of the monoidal category.
The extra wire that $\R$ provides is only used to prevent interchange, and so it does not really matter where it is placed in the input and the output.
We can choose to always place it on the left, for instance -- and indeed we will be able to do so -- but a better solution is to just consider objects ``up to some runtime braidings''.
This is formalized by the notion of \emph{braid clique}.

\begin{definition}[Braid clique]
  Given any list of objects $A_{0},\dots,A_{n}$ in $\obj{\hyV} = \obj{\hyG}$,
  we construct a \emph{clique}  \cite{trimble:coherence,shulman20182} in the category $\MONRUN(\hyV,\hyG)$: we consider the objects, $A_{0} \tensor \dots \tensor \R_{(i)} \tensor \dots \tensor A_{n}$, created by inserting the runtime $\R$ in all of the possible $0 \leq i \leq n+1$ positions; and we consider the family of commuting isomorphisms constructed by braiding the runtime,
  \[ \sigma_{i,j} \colon A_{0} \tensor \dots \tensor \R_{(i)} \tensor \dots \tensor A_{n} \to  A_{0} \tensor \dots \tensor \R_{(j)} \tensor \dots \tensor A_{n}.\]
  We call this the \emph{braid clique}, $\Braid(A_{0},\dots,A_{n})$, on that list.
\end{definition}

\begin{definition}
  A \emph{braid clique morphism}, $f \colon \Braid(A_{0},\dots,A_{n}) \to \Braid(B_{0},\dots,B_{m})$ is a family of morphisms in the runtime monoidal category, $\MONRUN(\hyV,\hyG)$, from each of the objects of first clique to each of the objects of the second clique,
  \[f_{ik} \colon A_{0} \tensor \dots \tensor \R_{(i)} \tensor \dots \tensor A_{n} \to
    B_{0} \tensor \dots \tensor \R_{(k)} \tensor \dots \tensor B_{m},\]
  that moreover commutes with all braiding isomorphisms, $f_{ij} \comp \sigma_{jk} = \sigma_{il} \comp f_{}$.
\end{definition}

A braid clique morphism  $f \colon \Braid(A_{0},\dots,A_{n}) \to \Braid(B_{0},\dots,B_{m})$ is fully determined by \emph{any} of its components, by pre/post-composing it with braidings.
In particular, a braid clique morphism is always fully determined by its leftmost component $f_{00} \colon \R \tensor A_{0} \tensor \dots \tensor A_{n} \to \R \tensor B_{0} \tensor \dots \tensor B_{m}$.

\begin{lemma}\label{lemma:eff-is-premonoidal}
  Let $(\hyV,\hyG)$ be a \polygraphCouple{}.
  There exists a \premonoidalCategory{}, $\EFF(\hyV,\hyG)$, that has objects the braid cliques, $\Braid(A_{0},\dots,A_{n})$, in $\MONRUN(\hyV,\hyG)$, and as morphisms the braid clique morphisms between them. 
  See Appendix.%
\end{lemma}

\begin{lemma}
  \label{lemma:identity-mon-eff}
  Let $(\hyV,\hyG)$ be a \polygraphCouple{}.
  There exists an identity-on-objects functor \(\MON(\hyV) \to \EFF(\hyV,\hyG)\) that strictly preserves the premonoidal structure and whose image is central. 
  See Appendix.%
\end{lemma}

\begin{lemma}
  \label{lemma:freeness}
  Let $(\hyV,\hyG)$ be a \polygraphCouple{} and consider the \effectfulCategory{} determined by \(\MON(\hyV) \to \EFF(\hyV,\hyG)\).
  Let $\cbaseV \to \ccatC$ be a strict effectful category endowed with a \polygraphCouple{} morphism $F \colon (\hyV,\hyG) \to \mathcal{U}(\cbaseV,\ccatC)$.
  There exists a unique strict effectful functor from $(\MON(\hyV) \to \EFF(\hyV,\hyG))$ to $(\cbaseV \to \ccatC)$ commuting with $F$ as a \polygraphCouple{} morphism. 
  See Appendix.%
\end{lemma}

\begin{theorem}[Runtime as a resource]
  \label{theorem:runtime-as-a-resource}
  The free strict effectful category over a \polygraphCouple{} $(\hyV,\hyG)$ is $\MON(\hyV) \to \EFF(\hyV,\hyG)$. Its morphisms $A \to B$ are in bijection with the morphisms $\R \tensor A \to \R \tensor B$ of the runtime monoidal category,
  \[\EFF(\hyV,\hyG)(A,B) \cong \MONRUN(\hyV,\hyG)(\R \tensor A, \R \tensor B).\]
\end{theorem}
\begin{proof}
  We must first show that $\MON(\hyV) \to \EFF(\hyV,\hyG)$ is an effectful category.
  The first step is to see that $\EFF(\hyV,\hyG)$ forms a premonoidal category (\Cref{lemma:eff-is-premonoidal}).
  We also know that $\MON(\hyV)$ is a monoidal category: in fact, a strict, freely generated one.
  There exists an identity on objects functor, \(\MON(\hyV) \to \EFF(\hyV,\hyG)\), that strictly preserves the premonoidal structure and centrality (\Cref{lemma:identity-mon-eff}).

  Let us now show that it is the free one over the \polygraphCouple{} $(\hyV,\hyG)$.
  Let $\cbaseV \to \ccatC$ be an effectful category, with an \polygraphCouple{} map $F \colon (\hyV,\hyG) \to \mathcal{U}(\cbaseV,\ccatC)$.
  We can construct a unique effectful functor from $(\MON(\hyV) \to \EFF(\hyV,\hyG))$ to $(\cbaseV \to \ccatC)$ giving its universal property (\Cref{lemma:freeness}).
\end{proof}

\begin{corollary}[String diagrams for effectful categories]
  We can use string diagrams for effectful categories, quotiented under the same isotopy as for monoidal categories, provided that we do represent the runtime as an extra wire that needs to be the input and output of every effectful morphism.
\end{corollary}

\section{Profunctors and Promonads}
\label{sec:promonads}

We have elaborated on string diagrams for \effectfulCategories{}.
Let us now show that \effectfulCategories{} are fundamental objects.
The profunctorial counterpart of a monad is a \promonad{}.
\Promonads{} have been widely used for functional programming semantics, although usually with an extra assumption of strength and under the name of ``arrows'' \cite{heunen06,hughes00,jacobs09}.
\Promonads{} over a category endow it with some new, ``effectful'', morphisms; while the base morphisms of the category are called the ``pure'' morphisms.
This terminology will coincide when regarding \effectfulCategories{} as promonads.

In this section, we introduce profunctors and promonads.
In the following sections, we show that \effectfulCategories{} are to \promonads{} what monoidal categories are to categories: they are the \pseudomonoids{} of a suitably constructed monoidal bicategory of promonads.
In order to obtain this result, we introduce the \pureTensor{} of \promonads{} in \Cref{sec:puretensor}.
The \pureTensor{} of \promonads{} combines the effects of two \promonads{} over different categories into a single one.
In some sense, it does so in the universal way that turns ``purity'' into ``centrality'' (\Cref{th:universalpure}).

\subsection{Profunctors: an algebra of processes}

\Profunctors{} $P \colon \catA^{op} \times \catB \to \Set$ \cite{benabou67,borceux94,benabou00} can be thought as indexing \emph{families of processes} $P(A,B)$ by the types of an input channel $A$ and an output channel $B$ \cite{monoidalstreams}.

The category $\catA$ has as morphisms the pure transformations $f \colon A' \to A$ that we can apply to the input of a process $p \in P(A,B)$ to obtain a new process, which we call $(f > p) \in P(A',B)$.
Analogously, the category $\catB$ has as morphisms the pure transformations $g \colon B \to B'$ that we can apply to the output of a process $p \in P(A,B)$ to obtain a new process, which we call $(p < g) \in P(A,B')$.
The \profunctor{} axioms encode the compositionality of these transformations.

\begin{definition}
  \defining{linkprofunctor}{}
  A \emph{profunctor} $(P,\lp,\rp)$ between two categories $\catA$ and $\catB$ is a family of sets $P(A,B)$ indexed by objects of $\catA$ and $\catB$, and endowed with jointly functorial left and right actions of the morphisms of $\catA$ and $\catB$, respectively.
  Explicitly, types of these actions are $\defining{linkleftprofunctoraction}{(\lp)} \colon \hom(A',A) \times P(A',B) \to P(A,B)$, and
  $\defining{linkrightprofunctoraction}{(\rp)} \colon \hom(B,B') \times P(A,B) \to P(A,B')$.
  They must satisfy
  \begin{itemize}
    \item compatibility, $(f \lp p) \rp g = f \lp (p \rp g)$,
    \item preserve identities, $\im \lp p = p$, and $p \rp \im = p$,
    \item and composition, $(p \rp f) \rp g = p \rp (f \comp g)$ and $f \lp (g \lp p) = (f \comp g) \lp p$.
  \end{itemize}
More succintly, a \emph{profunctor} $P \colon \catA \profarrow \catB$ is a functor $P \colon \catA^{op} \times \catB \to \Set$. When presented as a family of sets with a pair of actions, \profunctors{} are sometimes called \emph{bimodules}.
\end{definition}

A profunctor homomorphism $\alpha \colon P \to Q$ transforms processes of type $P(A,B)$ into processes of type $Q(A,B)$. The homomorphism affects only the effectful processes, and not the pure transformations we could apply in $\catA$ and $\catB$. This means that $\alpha(f \lp p) = f \lp \alpha(p)$ and that $\alpha(p \rp g) = \alpha(p) \rp g$.

\begin{definition}[Profunctor homomorphism]
  \defining{linkprofunctorhomomorphism}{}
  A \emph{profunctor homomorphism} from the \profunctor{} $P \colon \catA \profarrow \catB$ to the
  \profunctor{} $Q \colon \catA \profarrow \catB$ is a family of functions
  $\alpha_{A,B} \colon P(A,B) \to Q(A,B)$ preserving the left and right actions,
  $\alpha(f \lp p \rp g) = f \lp \alpha(p) \rp g$.
  Equivalently, it is a natural transformation $\alpha \colon P \to Q$ between the two functors $\catA^{op} \times \catB \to \Set$.
\end{definition}

How to compose two families of processes? Assume we have a process $p \in P(A,B_{1})$ and a process $q \in Q(B_{2},C)$.
Moreover, assume we have a transformation $f \colon B_{1} \to B_{2}$ translating from the output of the second to the input of the first.
In this situation, we can plug together the processes: $p \in P(A,B_{1})$ writes to an output of type $B_{1}$, which is translated by $f$ to an input of type $B_{2}$, then used by $q \in Q(B_{2},C)$.
There are two slightly different ways of describing this process, depending on whether we consider the translation to be part of the first or the second process.
We could translate just after finishing the first process, $(p \rp f, q)$; or translate just before starting the second process, $(p, f \lp q)$.

These are two different pairs of processes, with different types.
However, if we take the process interpretation seriously, it does not really matter when to apply the translation.
These two descriptions represent the same process.
They are \emph{dinaturally equivalent} \cite{monoidalstreams,loregian2021}.

\begin{definition}[Dinatural equivalence]
  Let $P \colon \catA \profarrow \catB$ and $Q \colon \catB \profarrow \catC$ be two \profunctors{}.
  Consider the set of matching pairs of processes, with a given input $A$ and output $C$,
  \[R_{P,Q}(A,C) = \sum_{B \in \catB} P(A,B) \times Q(B,C).\]
  \emph{Dinatural equivalence} $(\sim)$, on the set $R_{P,Q}(A,C)$ is the smallest equivalence relation satisfying
  $(p \rp g , q) \sim (p , g \lp q)$.
  The set of matching processes $R_{P,Q}(A,C)$ quotiented by dinaturality $(\sim)$ is written as $(P \diamond Q)(A,C)$.
  It is a particular form of colimit over the category $\catB$, called a \emph{coend}, usually denoted by an integral sign.
  \[(P \diamond Q)(A,C) = R_{P,Q}(A,C)/(\sim) = \int^{B \in \catB} P(A,B) \times Q(B,C).\]
\end{definition}

\begin{definition}[Profunctor composition]
  The composition of two profunctors $P \colon \catA \profarrow \catB$ and $Q \colon \catB \profarrow \catC$ is the profunctor $(P \profcomp Q) \colon \catA \profarrow \catC$ has as processes the matching pairs of processes in $P$ and $Q$ quotiented by dinaturality on $\catB$,
  \[(p,g \rp q) \sim (p \lp g, q).\]
  Its actions are the left and right actions of $p$ and $q$, respectively, $f \lp (p,q) \rp g = (f \lp p, q \rp g)$.

  The identity profunctor $\catA \colon \catA \profarrow \catA$ has as processes the morphisms of the category $\catA$, it is given by the hom-sets.
  Its actions are pre and post-composition, $f \lp h \rp g = f \comp h \comp g$.
\end{definition}

Profunctors are better understood as providing a double categorical structure to the category of categories.
A double category $\catD$ contains 0-cells (or ``objects''), two different types of 1-cells (the ``arrows'' and the ``proarrows''), and cells \cite{shulman2008framed}.
Arrows compose in an strictly associative and unital way, while proarrows come equipped with natural isomorphisms representing associativity and unitality.
We employ the graphical calculus of double categories \cite{myers16}, with arrows going left to right and proarrows going top to bottom.

\begin{definition}
  The double category of categories, $\mathbf{CAT}$, has as objects the small categories $\catA, \catB, \dots$, as arrows the functors between them, $F \colon \catA \to \catA'$, as proarrows the profunctors between them, $P \colon \catA \nrightarrow \catB$, and as cells, the natural transformations, $\alpha_{A,B} \colon P(A,B) \to Q(FA,GB)$.
  \begin{figure}[H]
    \centering
    \includegraphics[scale=0.6]{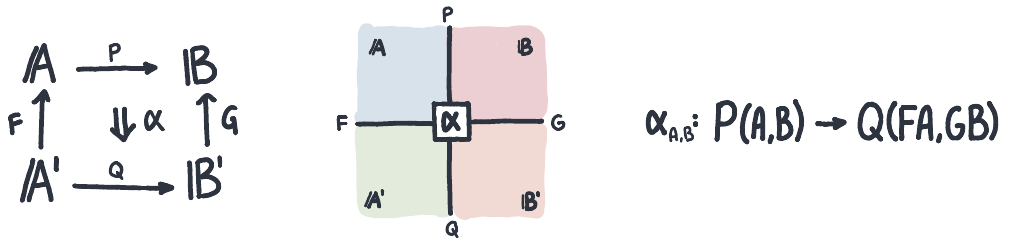}
    \caption{Cell in the double category of categories.}
    \label{fig:catcell}
  \end{figure}

  Every functor has a companion and a conjoint profunctors: their representable and corepresentable profunctors~\cite{grandis99}. This structure makes $\mathbf{CAT}$ into the paradigmatic example of a proarrow equipment (or \emph{framed bicategory}~\cite{shulman2008framed}).
\end{definition}

\subsection{Promonads: new morphisms for an old category}

Promonads are to profunctors what monads are to functors.\footnote{To quip, a promonad is just a monoid on the category of endoprofunctors.}
It may be then surprising to see that so little attention has been devoted to them, relative to their functorial counterparts.
The main source of examples and focus of attention has been the semantics of programming languages \cite{hughes00,paterson01,jacobs09}.
Strong monads are commonly used to give categorical semantics of effectful programs \cite{moggi91}, and the so-called \emph{arrows} (or \emph{strong promonads}) strictly generalize them.

Part of the reason behind the relative unimportance given to promonads elsewhere may stem from the fact that promonads over a category can be shown in an elementary way to be equivalent to identity-on-objects functors from that category~\cite{loregian2021}. The explicit proof is, however, difficult to find in the literature, and so we include it here (\Cref{th:promonadidonobjs}).

Under this interpretation, promonads are new morphisms for an old category. We can reinterpret the old morphisms into the new ones in a functorial way. The paradigmatic example is again that of Kleisli or cokleisli categories of strong monads and comonads.
This structure is richer than it may sound, and we will explore it further during the rest of this text.

\begin{definition}[Monoids and promonoids]
  A \emph{monoid} in a double category is an arrow $T \colon \catA \to \catA$ together with cells $m \in \hom(M \otimes M;1,1;M)$ and $e \in \cell(1;1,1;M)$, called multiplication and unit, satisfying unitality and associativity.
  A \emph{promonoid} in a double category is a proarrow $M \colon \catA \relto \catA$ together with cells $m \in \cell(1;M \otimes M,M,1)$ and $e \in \cell(1;1,M;1)$, called promultiplication and prounit, satisfying unitality and associativity.
  \begin{figure}[H]
    \centering
    \includegraphics[scale=0.6]{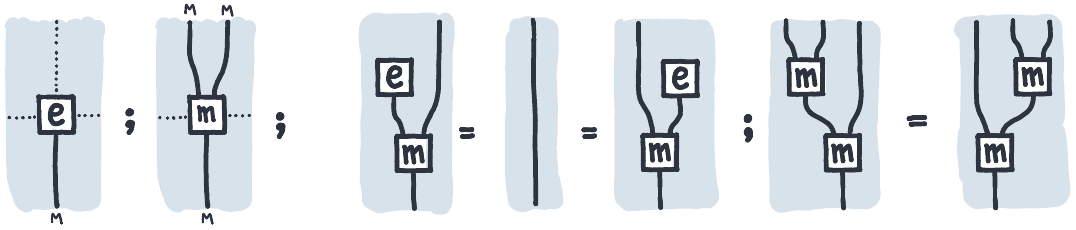}
    \caption{Data and axioms of a promonoid in a double category.}
  \end{figure}
  Dually, we can define \emph{comonoids} and \emph{procomonoids}.
\end{definition}

A monad is a monoid in the category of categories, functors and profunctors $\CAT$.
In the same way, a promonad is a promonoid in $\CAT$.

\begin{definition}\label{definition:promonad}
  \defining{linkpromonad}{}
  A \emph{promonad} $(P,\starp,\unitp{})$ over a category $\catC$ is a profunctor $P \colon \catC \profarrow \catC$ together with natural transformations representing inclusion $(\unitp{})_{X,Y} \colon \catC(X,Y) \to P(X,Y)$ and multiplication $(\defining{linkpromonadmultiplication}{\starp})_{X,Y} \colon P(X,Y) \times P(Y,Z) \to P(X,Z)$, and such that
  \begin{enumerate}[label=\roman*.]
    \item the right action is premultiplication, $\unitp{f} \starp p = f \lp p$;
    \item the left action is posmultiplication, $p \starp \unitp{f} = p \rp f$;
    \item multiplication is dinatural, $p \starp (f \lp q) = (p \rp f) \starp q$;
    \item and multiplication is associative, $(p_{1} \starp p_{2}) \starp p_{3} = p_{1} \starp (p_{2} \starp p_{3})$.
  \end{enumerate}
  Equivalently, promonads are promonoids in the double category of categories, where the
dinatural multiplication represents a transformation from the composition of the profunctor $P$ with itself.
\end{definition}

\begin{lemma}[Kleisli category of a promonad]
  \label{lemma:kleisli}
  Every promonad $(P,\starp,\unitp{})$ induces a category with the same objects as its base category, but with hom-sets given by $P(\bullet,\bullet)$, composition given by $(\starp)$ and identities given by $(\unitp{\im})$.
  This is called its \emph{Kleisli category}, $\kleisli{P}$. Moreover, there exists an identity-on-objects functor $\catC \to \kleisli{P}$, defined on morphisms by the unit of the promonad. 
  See Appendix.%
\end{lemma}

The converse is also true: every category $\catC$ with an identity-on-objects functor from some base category $\catV$ arises as the Kleisli category of a promonad.

\begin{theorem}
  \label{th:promonadidonobjs}
  Promonads over a category $\catC$ correspond to identity-on-objects functors from the category $\catC$.
  Given any identity-on-objects functor $i \colon \catC \to \catD$ there exists a unique promonad over $\catC$ having $\catD$ as its Kleisli category: the promonad given by the profunctor $\hom_{\catD}(i(\bullet),i(\bullet))$. 
  See Appendix.%
\end{theorem}

\subsection{Homomorphisms and transformations of promonads}

We have characterized promonads as identity-on-objects functors.
We now characterize the homomorphisms and transformations of promonads as suitable pairs of functors and natural transformations.

\begin{definition}[Promonoid homomorphism]
  Let $(\catA,M,m,e)$ and $(\catB,N,n,u)$ be promonoids in a double category.
  A promonoid homomorphism is an arrow $T \colon \catA \to \catB$ together with a cell $t \in \cell(F;M,N;F)$ that preserves the promonoid promultiplication and prounit.
  \begin{figure}[H]
    \centering
    \includegraphics[scale=0.6]{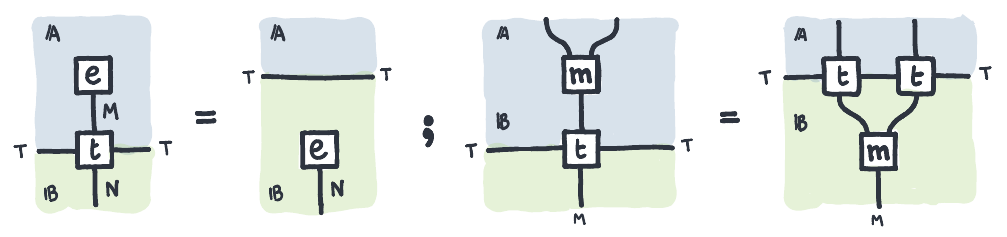}
    \caption{Axioms for a promonoid homomorphism.}
  \end{figure}
\end{definition}

\begin{definition}[Promonad homomorphism]
  \defining{linkpromonadhomomorphism}{}
  Let $(\catA,P,\starp{},\unitp{})$ and $(\catB,Q,\starp{},\unitp{})$ be two \promonads{}, possibly over two different categories.
  A \promonadHomomorphism{} $(F_{0},F)$ is a functor between the underlying categories $F_{0} \colon \catA \to \catB$ and a natural transformation $F_{X,Y} \colon P(X,Y) \to Q(FX,FY)$ preserving composition and inclusions.
  That is, $F(p_{1} \starp p_{2}) = F(p_{1}) \starp F(p_{2})$, and $F(\unitp{f}) = \unitp{F_{0}(f)}$.
\end{definition}

\begin{proposition}
  \label{prop:homomorphismcommuting}
  A \promonadHomomorphism{} between two \promonads{} understood as identity-on-objects functors, $\baseV \to \catC$ and $\baseW \to \catD$, is equivalently a pair of functors $(F_{0},F)$ that commute strictly with the two identity-on-objects functors on objects $F_{0}(X) = F(X)$ and morphisms $\unitp{F_{0}(f)} = F(\unitp{f})$.
  See Appendix.%
\end{proposition}

\begin{definition}[Promonoid modification]
  \defining{linkpromonoidmodification}{}
  Let $(\catA,M,m,e)$ and $(\catB,N,n,u)$ be promonoids in a double category,
  and let $t \in \cell(F;M,N;F)$ and $r \in \cell(G;M,N;G)$ be promonoid homomorphisms.
  A \emph{promonoid modification} is a cell $\alpha \in \cell(F;1,1;G)$ such that its precomposition with $t$ is its postcomposition with $r$.
  \begin{figure}[H]
    \centering
    \includegraphics[scale=0.6]{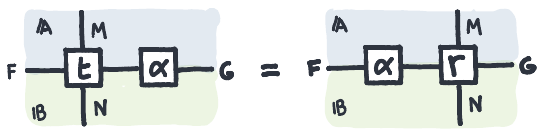}
    \caption{Axiom for a promonoid transformation.}
  \end{figure}
\end{definition}

\begin{definition}
  \defining{linkpromonadmodification}{}
  A \emph{\promonadModification{}} between two \promonadHomomorphisms{} $(F_{0},F)$ and $(G_{0},G)$ between the same \promonads{} $(\catA,P,\starp,\unitp{})$ and $(\catB,Q,\starp,\unitp{})$ is a natural transformation $\alpha_{X} \colon F_{0}(X) \to G_{0}(X)$ such that $\alpha_{X} \lp G(p) = F(p) \rp \alpha_{Y}$ for each $p \in P(X,Y)$.
\end{definition}
\begin{proposition}
  \label{prop:modificationcylinder}
  A \promonadModification{} between two \promonadHomomorphisms{} understood as commutative squares of identity-on-objects functors $\unitp{F_{0}(f)} = F(\unitp{f})$ and  $\unitp{G_{0}(f)} = G(\unitp{f})$ is a natural transformation $\alpha \colon F_{0} \Rightarrow G_{0}$ that can be lifted via the identity-on-objects functor to a natural transformation $\unitp{\alpha} \colon F \Rightarrow G$.
  In other words, a pure natural transformation.

  \begin{figure}[H]
    \centering
    \includegraphics[scale=0.6]{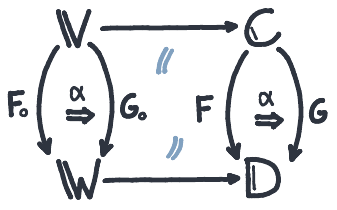}
    \caption{Promonad modifications are cylinder transformations.}
    \label{fig:cylinder}
  \end{figure}

\end{proposition}

Summarizing this section, we have shown a correspondence between promonads, their homomorphisms and modifications, and identity-on-objects functors, squares and cylinder transformations of squares.
The double category structure allows us to talk about homomorphisms and modifications, which would be more difficult to address in a bicategory structure.
\begin{center}
\begin{tabular}{ c|c|c }
 \Promonad{} & Identity-on-objects functor & \Cref{th:promonadidonobjs} \\ \hline
 \PromonadHomomorphism{} & Commuting square & \Cref{prop:homomorphismcommuting} \\ \hline
 \PromonadModification{} & Cylinder transformation & \Cref{prop:modificationcylinder}
\end{tabular}
\end{center}

\section{Pure Tensor of Promonads}
\label{sec:puretensor}

This section introduces the \emph{pure tensor of promonads}.
The \pureTensor{} of \promonads{} combines the effects of two \promonads{}, possibly over different categories, into the effects of a single \promonad{} over the product category.
Effects do not generally interchange.
However, this does not mean that no morphisms should interchange in the \pureTensor{} of \promonads{}:
in our interpretation of a \promonad{} $\baseV \to \catC$, the morphisms coming from the inclusion are \emph{pure}, they produce no effects; pure morphisms with no effects should always interchange with effectful morphisms, even if effectful morphisms do not interchange among themselves.

A practical way to encode and to remember all of the these restrictions is to use monoidal string diagrams.
This is another application of the idea of \colorPre{runtime}: we introduce an extra wire so that all the rules of interchange become ordinary interchange laws in a monoidal category.
That is, we insist again that effectful morphisms are just pure morphisms using a shared resource -- the \colorPre{runtime}.
When we compute the \pureTensor{} of two promonads, the runtime needs to be shared between the impure morphisms of both promonads.

\subsection{Pure tensor, via runtime}
\begin{definition}[Pure tensor]
  Let $\catC \colon \baseV \profarrow \baseV$ and $\catD \colon \baseW \profarrow \baseW$ be two \promonads{}.
  Their \emph{pure tensor}, $\catC \ast \catD \colon \baseV \times \baseW \to \baseV \times \baseW$, is a promonad over $\baseV \times \baseW$ where elements of $\catC \ast \catD (X,Y;X',Y')$, the morphisms $X \tensor \R \tensor Y \to X' \tensor \R \tensor Y'$ in the freely presented monoidal category generated by the elements of \Cref{fig:runtimeprop} and quotiented by the axioms of \Cref{fig:sharedaxioms}.
  \begin{figure}[H]
    \centering
    \includegraphics[scale=0.6]{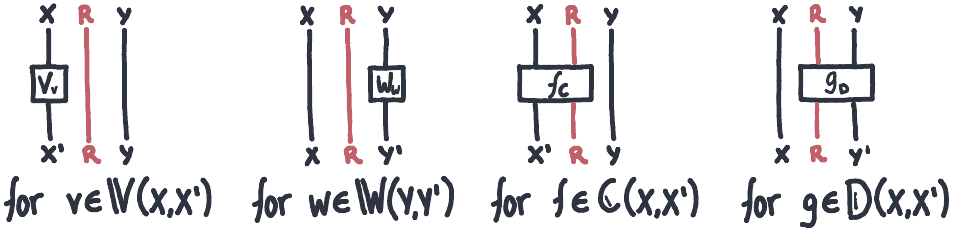}
    \caption{Generators for the elements of the pure tensor of promonads.}
    \label{fig:runtimeprop}
  \end{figure}

  \begin{figure}
    \centering
    \includegraphics[scale=0.6]{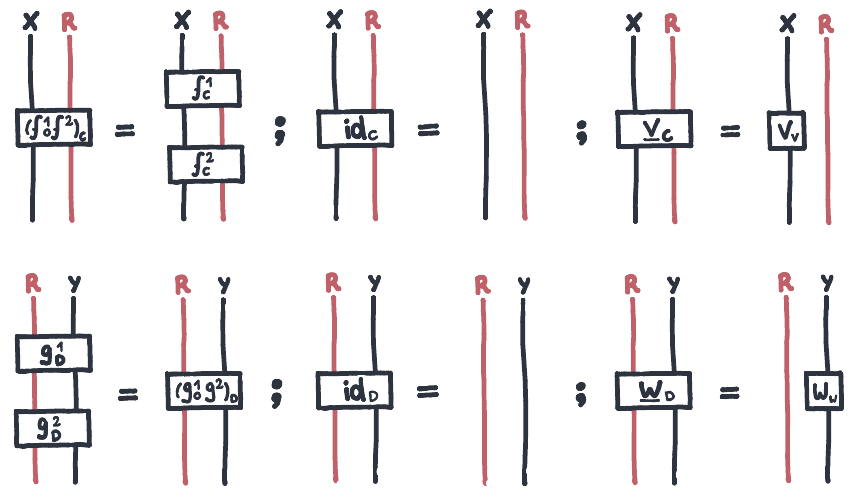}
    \caption{Axioms for the elements of the pure tensor of promonads.}
    \label{fig:sharedaxioms}
  \end{figure}
  Multiplication is defined by composition in the monoidal category, and the unit is defined by the inclusion of pairs, as depicted in \Cref{fig:puretensor}.
  \begin{figure}[H]
    \centering
    \includegraphics[scale=0.6]{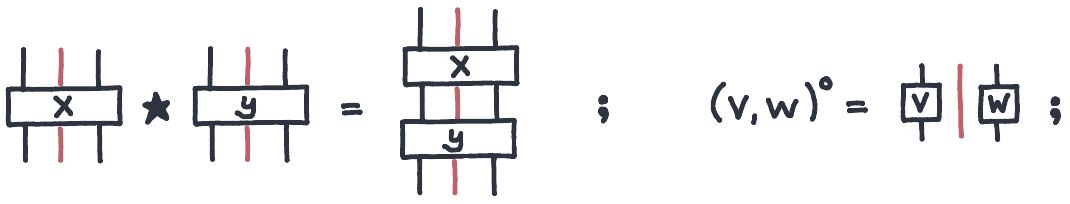}
    \caption{The pure tensor promonad.}
    \label{fig:puretensor}
  \end{figure}
  In other words, the elements of the pure tensor are the morphisms the category presented by the graph that has as objects the pairs of objects $(X,Y)$ with $X \in \obj{\baseV}$ and $Y \in \obj{\baseW}$, formally written as $X \tensor \R \tensor Y$; and the morphisms generated by
  \begin{itemize}
    \item an edge $f_{\catC} \colon X \tensor \R \tensor Y \to X' \tensor \R \tensor Y$ for each arrow $f \in \catC(X,X')$  and each object $Y \in \baseW$;
    \item an edge $g_{\catD} \colon X \tensor \R \tensor Y \to X \tensor \R \tensor Y'$ for each arrow $g \in \catD(Y,Y')$  and each object $X \in \baseV$;
    \item an edge $v_{\baseV} \colon X \tensor \R \tensor Y \to X' \tensor \R \tensor Y$ for each arrow $v \in \baseV(X,X')$  and each object $Y \in \baseW$;
    \item and an edge $w_{\baseW} \colon  X \tensor \R \tensor Y \to X \tensor \R \tensor Y'$ for each arrow $w \in \baseW(Y,Y')$  and each object $X \in \baseV$;
  \end{itemize}
  quotiented by centrality of pure morphisms: $f_{\catC} \comp w_{\baseW} = w_{\baseW} \comp f_{\catC}$ and $g_{\catD} \comp v_{\baseV} = v_{\baseV} \comp g_{\catD}$; by compositions and identities of one promonad: $f_{\catC} \comp f'_{\catC} = (f \starp f')_{\catC}$ and $\im_{\catC} = \im$; by compositions and identities of the other promonad: $g_{\catD} \comp g'_{\catD} = (g \starp g')_{\catD}$ and $\im_{\catD} = \im$; and by the coincidence of pure morphisms and their effectful representatives: $v_{\baseV} = \unitp{v}_{\catC}$ and $w_{\baseW} = \unitp{w}_{\catD}$.
\end{definition}

Crucially in this definition, $f_{\catC}$ and $g_{\catD}$ do not interchange: they are sharing the \colorPre{runtime}, and that prevents the application of the interchange law.
The \pureTensor{} of promonads, $\catC \ast \catD$, takes its name from the fact that, if we interpret the promonads $\baseV \to \catC$ and $\baseW \to \catD$ as declaring the morphisms in $\baseV$ and $\baseW$ as pure, then the pure morphisms of the composition interchange with all effectful morphisms.
The spirit is similar to the \emph{free product of groups with commuting subgroups} \cite{magnus2004combinatorial}.

\subsection{Universal property of the pure tensor}

There are multiple canonical ways in which one could combine the effects of two promonads,
$\catC \colon \baseV \profarrow \baseV$ and $\catD \colon \baseW \profarrow \baseW$, into a single promonad, such as taking the product of both, $\catC \times \catD \colon \baseV \times \baseW \profarrow \baseV \times \baseW$.
Let us show that the \pureTensor{} has a universal property: it is the universal one in which we can include impure morphisms from each promonads, interchanging with pure morphisms from the other promonad, so that purity is preserved.

\begin{theorem}
  \label{th:universalpure}
  Let $\catC \colon \baseV \profarrow \baseV$ and $\catD \colon \baseW \profarrow \baseW$ be two \promonads{} and let $\catC \ast \catD \colon \baseV \times \baseW \to \baseV \times \baseW$ be their pure tensor.
  There exist a pair of \promonadHomomorphisms{} $L \colon \catC \times \baseW \to \catC \ast \catD$ and $R \colon \baseV \times \catD \to \catC \ast \catD$. These are universal in the sense that, for every pair of promonad homomorphisms, $A \colon \catC \times \baseW \to \catE$ and $B \colon \baseV \times \catD \to \catE$, there exists a unique promonad homomorphism $(A \vee B) \colon \catC \ast \catD \to \catE$ that commutes strictly with them, $(A \vee B) \comp L = A$ and $(A \vee B) \comp R = B$. 
  See Appendix.%
\end{theorem}

\section{Effectful Categories are Pseudomonoids}
\label{sec:pseudomonoids}
We will now use the \pureTensor{} of \promonads{} to justify \effectfulCategories{} as the promonadic counterpart of monoidal categories: \effectfulCategories{} are \pseudomonoids{} in the monoidal bicategory of \promonads{} with the \pureTensor{}.
\Pseudomonoids{} \cite{street97,verdon17} are the categorification of monoids.
They are still formed by a 0-cell representing the carrier of the monoid and a pair of 1-cells representing multiplication and units. However, we weaken the requirement for associativity and unitality to the existence of invertible 2-cells, called the \emph{associator} and \emph{unitor}.

In the same way that monoids live in monoidal categories, pseudomonoids live in monoidal bicategories.
A monoidal bicategory $\bicatA$ is a bicategory in which we can tensor objects with a pseudofunctor $(\boxtimes) \colon \bicatA \times \bicatA \to \bicatA$ and we have a tensor unit $I \colon 1 \to \bicatA$, these are associative and unital up to equivalence, and satisfy certain coherence equations up to invertible modification \cite{schommerpries2011classification}.

\subsection{Pseudomonoids}
\begin{definition}
  In a monoidal bicategory, a \emph{pseudomonoid} over a 0-cell $M$ is a pair of 1-cells, $M \boxtimes M \to M$ and $I \to M$, together with the following triple of invertible 2-cells representing associativity and unitality (\Cref{fig:pseudomonoid}), and satisfying the pentagon and triangle equations 
  (see Appendix).%
  A homomorphism of pseudomonoids is given by a 1-cell between their underlying 0-cells and the following invertible 2-cells, representing preservation of the multiplication and the unit (\Cref{fig:pseudomonoidhom}), and satisfying compatibility with associativity and unitality 
  (see Appendix).%
  \begin{figure}[H]
    \centering
    \includegraphics[scale=0.6]{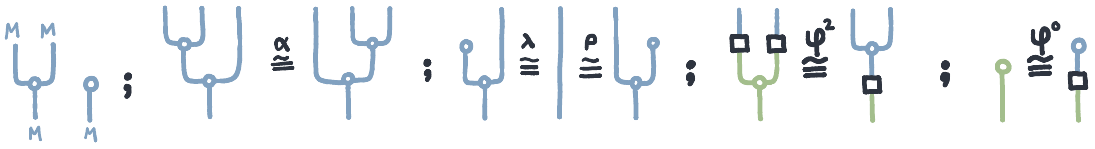}
    \caption{Data for a pseudomonoid and pseudomonoid homomorphism.}
    \label{fig:pseudomonoid}
    \label{fig:pseudomonoidhom}
  \end{figure}
\end{definition}

  A \pseudomonoid{} is \emph{strict} when the associators and unitors are identity cells.
  Note that, in strict 2-categories (sometimes called 2-categories, in contrast to bicategories), this is the same as a monoid in the monoidal category that we obtain by ignoring the 2-cells.
  \begin{remark}
    \label{remark:funnydoesnotwork}
  A \pseudomonoid{} in the monoidal bicategory of categories with the cartesian product of categories, $(\CAT, \times)$ is a monoidal category.
  A strict \pseudomonoid{} in the same monoidal bicategory is a strict monoidal category.
\end{remark}
A strict \pseudomonoid{} in the monoidal bicategory of categories with the funny tensor product of categories $(\CAT, \square)$ is a strict premonoidal category.
However, it is not immediately clear how to recover premonoidal categories as \pseudomonoids{}.
A naive attempt will fail: $(\CAT, \square)$ is usually made into a monoidal bicategory with non-necessarily-natural transformations, but we do want our coherence morphisms to be natural, so we must ask at least naturality.
This will not be enough: taking natural transformations as 2-cells will give us premonoidal categories where the associators and unitors do not need to be \emph{central}.
Centrality is what requires a more careful approach.

\subsection{Effectful categories are promonad pseudomonoids}

Promonads form a monoidal category with the pure tensor product and moreover a strict monoidal bicategory with promonad modifications.
Effectful categories are the pseudomonoids in this category.

\begin{theorem}
  \label{th:freydpseudomonoid}
  An \effectfulCategory{} (or monoidal Freyd category) is a pseudomonoid on the monoidal 2-category of promonads with promonad homomorphism, promonad transformations and the \pureTensor{} of promonads.
  A pseudomonoid homomorphism between \effectfulCategories{} is an \effectfulFunctor{}.

  As a consequence, preomonoidal categories \emph{with their centre} are pseudomonoids.
  See Appendix.%
\end{theorem}

\section{Conclusions}

Premonoidal categories are monoidal categories with runtime, and we can stil use monoidal string diagrams and unrestricted topological deformations to reason about them. 
Instead of dealing directly with premonoidal categories, we employ the better behaved notion of non-cartesian Freyd categories, \effectfulCategories{}. There exists a more fine-grained notion of ``Cartesian effect category'' \cite{dumas11}, which generalizes Freyd categories
and justifies calling “effectful category” to the general case.

Promonads have been arguably under-appreciated, possibly because of their characterization as ``just'' identity-on-objects functors.
However, speaking of promonads as the proarrow counterpart of monads makes many aspects of the theory of monads clearer: every monad and every comonad induce a promonad (their Kleisli category) via the proarrow equipment, monad morphisms lift to promonad morphisms, distributive laws of monads induce a way of composing morphisms from different kleisli categories \cite{cheng21}.
Justifying \effectfulCategories{} in terms of promonads highlights their importance as the monadic counterpart of monoidal categories.

Ultimately, this is a first step towards our more ambitious project of presenting the categorical structure of programming languages in a purely diagrammatic way, revisiting Alan Jeffrey's work \cite{jeffrey97,jeffrey1997premonoidal,marionotes2022}.
The internal language of premonoidal categories and effectful categories is given by the \emph{arrow do-notation} \cite{paterson01}; at the same time, we have shown that it is given by suitable string diagrams.
This correspondence allows us to translate between programs and string diagrams (\Cref{fig:premonoidalprogram}).

\begin{figure}[H]
  \centering
  \includegraphics[scale=0.7]{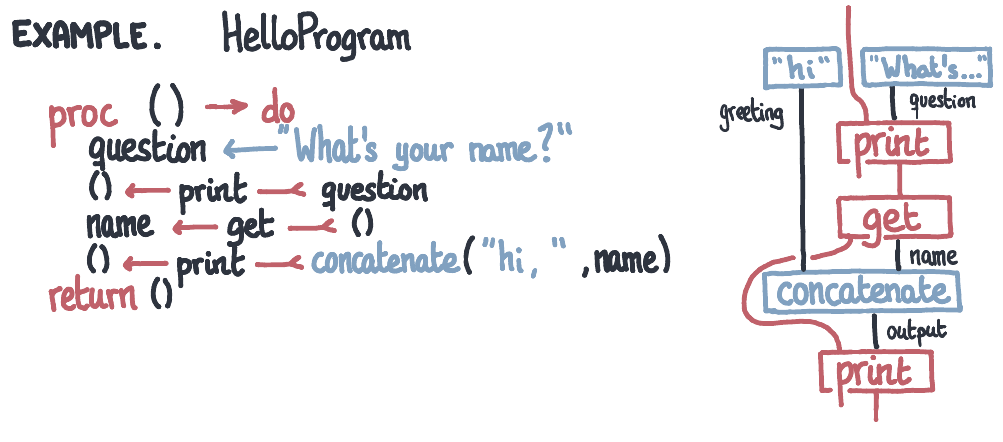}
  \caption{Premonoidal program in arrow do-notation and string diagrams.}
  \label{fig:premonoidalprogram}
\end{figure}

\paragraph{Related work.}
Staton and Møgelberg \cite{mogelberg14} propose a formalization of Jeffrey's graphical calculus for effectful categories that arise as the Kleisli category of a strong monad. 
They prove that \emph{'every strong monad is a linear-use state monad'}, that is, a state monad of the form $\R \multimap !(\bullet) \otimes \R$, where the state $\R$, is an object that cannot be copied nor discarded.

\section{Acknowledgements}
The author wants to thank Tarmo Uustalu, Sam Staton, Niels Voorneveld and Paweł Sobociński for many helpful comments and discussion. The author gratefully thanks the anonymous reviewers at ACT'22 for many constructive suggestions that improved this manuscript.
Mario Román was supported by the European Union through the ESF funded Estonian IT Academy research measure (2014-2020.4.05.19-0001); this work was also supported by the Estonian Research Council grant PRG1210.

\newpage

\bibliographystyle{eptcs}
\bibliography{bibliography}

\newpage

\appendix
\section{Effectful string diagrams}

During the following two lemmas, we will choose to always deal with the leftmost component of the braid clique morphism.
Given any clique $\Braid(A_{0},\dots,A_{n})$ we call $A = \tensordot{A}{n}$ to its tensoring;
clique morphisms $\Braid(A_{0},\dots,A_{n}) \to \Braid(B_{0},\dots,B_{m})$ are represented by morphisms $\R \tensor A \to \R \tensor B$.

\begin{lemma}\label{ax:lemma:eff-is-premonoidal}
  Let $(\hyV,\hyG)$ be a \polygraphCouple{}.
  There exists a \premonoidalCategory{}, $\EFF(\hyV,\hyG)$, that has as objects the braid cliques, $\Braid(A_{0},\dots,A_{n})$, in $\MONRUN(\hyV,\hyG)$, and as morphisms the braid clique morphisms between them.
\end{lemma}
\begin{proof}
  Let us first give $\EFF(\hyV,\hyG)$ category structure.
  The identity on $\Braid(A_{0},\dots,A_{n})$ is the identity on $\R \tensor A$.
  The composition of a morphism $\R  \tensor A \to \R \tensor B$ with a morphism $\R \tensor B \to \R \tensor C$ is their plain composition in $\MONRUN(\hyV,\hyG)$.

  Let us now check that it is moreover a \premonoidalCategory{}.
  Tensoring of cliques is given by concatenation of lists, which coincides with the tensor in $\MONRUN(\hyV,\hyG)$. However, it is interesting to note that the tensor of morphisms cannot be defined in this way:
  a morphism $\R \tensor A \to \R \tensor B$ cannot be tensored with a morphism $\R \tensor A' \to \R \tensor B'$ to obtain a morphism $\R \tensor A \tensor A' \to \R \tensor B \tensor B'$.

  Whiskering of a morphism $f \colon \R  \tensor A \to \R \tensor B$ is defined with braidings in the left case, $\R \tensor C \tensor A \to \R \tensor C \tensor B$, and by plain whiskering in the right case, $\R \tensor A \tensor C \to \R \tensor B \tensor C$, as depicted in \Cref{fig:whiskering}.
   \begin{figure}[H]
    \centering
    \includegraphics[scale=0.6]{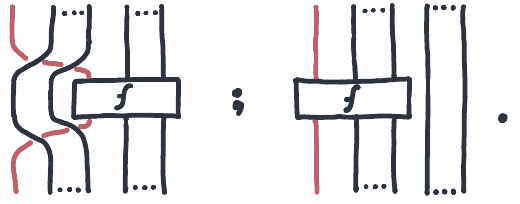}
    \caption{Whiskering in the runtime premonoidal category.}
    \label{fig:whiskering}
  \end{figure}

  Finally, the associators and unitors are identities, which are always natural and central.
\end{proof}

\begin{lemma}
  \label{ax:lemma:identity-mon-eff}
  Let $(\hyV,\hyG)$ be a \polygraphCouple{}.
  There exists an identity-on-objects functor \(\MON(\hyV) \to \EFF(\hyV,\hyG)\) that strictly preserves the premonoidal structure and whose image is central. This determines an effectful category.
\end{lemma}
\begin{proof}
  A morphism $v \in \MON(\hyV)(A,B)$ induces a morphism $(\im_{\R} \tensor v) \in \MONRUN(\hyV,\hyG)(\R \tensor A,\R \tensor B)$, which can be read as a morphism of cliques $(\im_{\R} \tensor v) \in \EFF(\hyV,\hyG)(A,B)$.
  This is tensoring with an identity, which is indeed functorial.

  Let us now show that this functor strictly preserves the premonoidal structure.
  The fact that it preserves right whiskerings is immediate.
  The fact that it preserves left whiskerings follows from the axioms of symmetry (\Cref{fig:whiskeringeq}, left).
  Associators and unitors are identities, which are preserved by tensoring with an identity.
  \begin{figure}[H]
    \centering
    \includegraphics[scale=0.6]{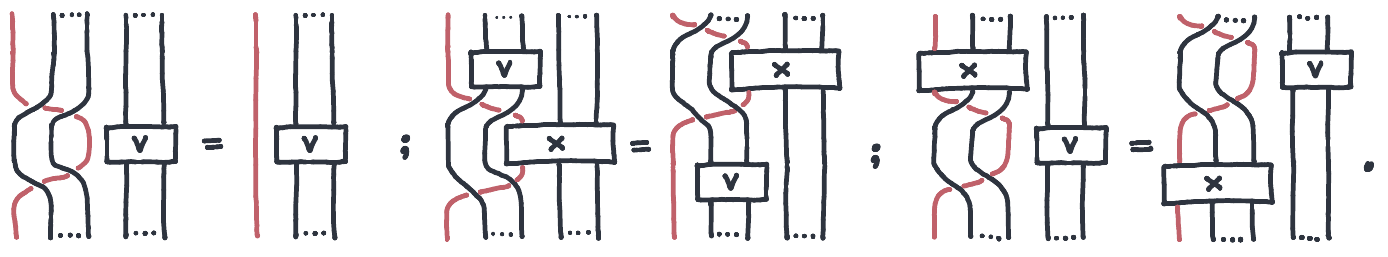}
    \caption{Preservation of whiskerings, and centrality.}
    \label{fig:centrality}    \label{fig:whiskeringeq}
  \end{figure}
  Finally, we can check by string diagrams that the image of this functor is central, interchanging with any given $x \colon \R \tensor C \to \R \tensor D$ (\Cref{fig:centrality}, center and right).
\end{proof}

\begin{lemma}[Freeness]
  \label{ax:lemma:freeness}
  Let $(\hyV,\hyG)$ be a \polygraphCouple{} and consider the effectful category determined by \(\MON(\hyV) \to \EFF(\hyV,\hyG)\).
  Let $\cbaseV \to \ccatC$ be a strict effectful category, with a \polygraphCouple{} morphism $F \colon (\hyV,\hyG) \to \mathcal{U}(\cbaseV,\ccatC)$.
  There exists a unique effectful functor from $(\MON(\hyV) \to \EFF(\hyV,\hyG))$ to $(\cbaseV \to \ccatC)$ commuting with $F$ as a \polygraphCouple{} morphism.
\end{lemma}
\begin{proof}
  By freeness, there already exists a unique strict monoidal functor $H_{0} \colon \MON(\hyV) \to \cbaseV$ that sends any object $A \in \obj{\hyV}$ to $F_{obj}(A)$.
  We will show there is a unique way to extend this functor together with the hypergraph assignment $\hyG \to \catC$ into a functor $H \colon \EFF(\hyV,\hyG) \to \ccatC$.
  Giving such a functor amounts to give some mapping of morphisms containing the runtime $\R$ in some position in their input and output,
  \[f \colon A_{0} \tensor \dots \tensor \R \tensor \dots \tensor A_{n} \to B_{0} \tensor \dots \tensor \R \tensor \dots \tensor B_{m}\]
  to morphisms $H(f) \colon FA_{0} \tensor \dots  \tensor FA_{n} \to FB_{0} \tensor \dots  \tensor FB_{n}$ in $\ccatC$, in a way that preserves composition, whiskerings, inclusions from $\MON(\hyV)$, and that is invariant to composition with braidings.
  In order to define this mapping, we will perform structural induction over the monoidal terms of the runtime monoidal category of the form
  $\MONRUN(\hyV,\hyG)(A_{0} \tensor \dots \tensor \R^{(i)} \tensor \dots \tensor A_{n}, \R \tensor B_{0} \tensor \dots  \tensor \R^{(j)} \tensor  \dots \tensor B_{m})$ and show that it is the only mapping with these properties (\Cref{fig:assignment}).
  
  Monoidal terms in a strict, freely presented, monoidal category are formed by identities ($\im$), composition $(\comp)$, tensoring $(\otimes)$, and some generators (in this case, in \Cref{fig:rungen}). Monoidal terms are subject to \emph{(i)} functoriality of the tensor, $\im \tensor \im = \im$ and $(f \comp g) \tensor (h \comp k) = (f \tensor h) \comp (g \tensor k)$; \emph{(ii)} associativity and unitality of the tensor, $f \tensor \im_{I} = f$ and $f \tensor (g \tensor h) = (f \tensor g) \tensor h$; \emph{(iii)} the usual unitality, $f \comp \im = f$ and $\im \comp f = f$ and associativity $f \comp (g \comp h) = (f \comp g) \comp h$; \emph{(iv)} the axioms of our presentation (in this case, in \Cref{fig:runaxiom}).
  \begin{figure}[H]
    \centering
    \includegraphics[scale=0.60]{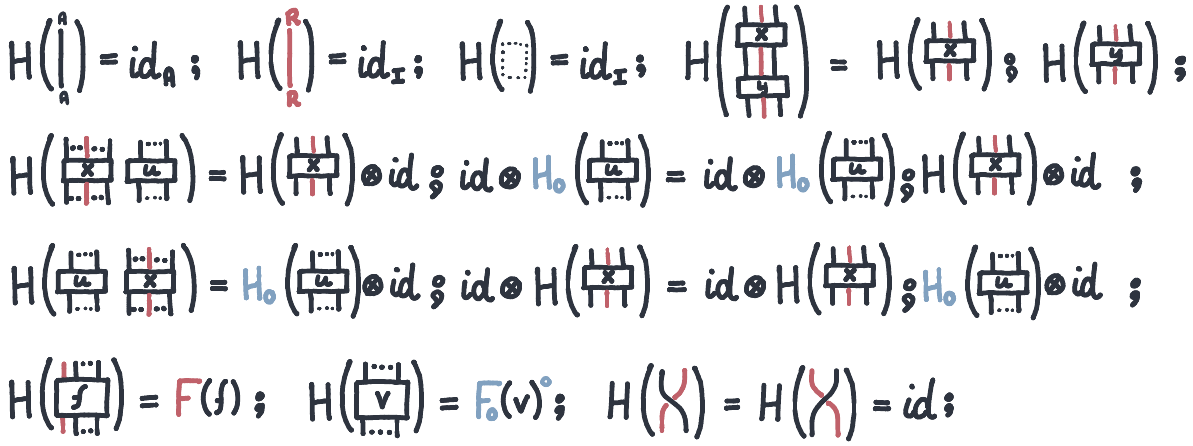}
    \caption{Assignment on morphisms, defined by structural induction on terms.}
    \label{fig:assignment}
  \end{figure}

  \begin{itemize}
    \item
      If the term is an identity, it can be \emph{(i)} an identity on an object $A \in \obj{(\hyV,\hyG)}$, in which case it must be mapped to the same identity by functoriality, $H(\im_{A}) = \im_{A}$; \emph{(ii)} an identity on the runtime, in which case it must be mapped to the identity on the unit object, $H(\im_{\R}) = \im_{I}$; or \emph{(iii)} an identity on the unit object, in which case it must be mapped to the identity on the unit, $H(\im_{I}) = \im_{I}$.

    \item
      If the term is a composition, $(f \comp g) \colon A_{0} \tensor \dots \tensor \R \tensor \dots \tensor A_{n} \to C_{0} \tensor \dots \tensor \R \tensor \dots \tensor C_{k}$, it must be along a boundary of the form $B_{0} \tensor \dots \tensor \R \tensor \dots \tensor B_{m}$: this is because every generator leaves the number of runtimes, $\R$, invariant.
      Thus, each one of the components determines itself a braid clique morphism.
      We must preserve composition of braid clique morphisms, so we must map $H(f \comp g) = H(f) \comp H(g)$.

    \item
      If the term is a tensor of two terms,  $(x \tensor u) \colon A_{0} \tensor \dots \tensor \R \tensor \dots \tensor A_{n} \to B_{0} \tensor \dots \tensor \R \tensor \dots \tensor B_{m}$, then only one of them was a term taking $\R$ as input and output (without loss of generality, assume it to be the first one) and the other was not: again, by construction, there are no morphisms taking one $\R$ as input and producing none, or viceversa.
      We split this morphism into $x \colon A_{0} \tensor\dots\tensor \R \tensor \dots \tensor A_{i-1} \to B_{0} \tensor \dots \tensor \R \tensor \dots \tensor B_{j-1}$ and $u \colon A_{i} \tensor \dots \tensor A_{n} \to B_{j} \tensor \dots \tensor B_{m}$.

      Again by structural induction, this time over terms $u \colon A_{i} \tensor \dots \tensor A_{n} \to B_{j} \tensor \dots \tensor B_{m}$,
      we know that the morphism must be either a generator in $\hyV(A_{i},\dots,A_{n};B_{j},\dots,B_{n})$ or a composition and tensoring of them. That is, $u$ is a morphism in the image of $\MON(\hyV)$, and it must be mapped according to the functor $H_{0} \colon \MON(\hyV) \to \cbaseV$.

      By induction hypothesis, we know how to map the morphism $x \colon A_{0} \tensor \dots\tensor  \R \tensor \dots \tensor A_{i-1} \to B_{0} \tensor \dots \tensor \R \tensor \dots \tensor B_{j-1}$.
      This means that, given any tensoring $x \tensor u$, we must map it to $H(x \tensor u) = (H(x) \tensor \im) \comp (\im \tensor H_{0}(u)) =  (\im \tensor H_{0}(u)) \comp (H(x) \tensor \im)$, where $H_{0}(u)$ is central.

    \item
      If the string diagram consists of a single generator, $f \colon \R \tensor A \to \R \tensor B$, it can only come from a generator $f \in \Run(\hyV,\hyG)(\R,A_{0},\dots,A_{n};\R, B_{0},\dots,B_{m}) = \hyG(A_{0},\dots,A_{n}; B_{0},\dots,B_{m})$,
      which must be mapped to $H(f) = F(f) \in \ccatC(A_{0} \tensor \dots \tensor A_{n}, B_{0} \tensor \dots \tensor B_{m})$.
      If the string diagram consists of a single braiding, it must be mapped to the identity, because the want the assignment to be invariant to braidings.

  \end{itemize}

  Now, we need to prove that this assignment is well-defined with respect to the axioms of these monoidal terms. Our reasoning follows \Cref{fig:assignmentwell}.

  \begin{figure}
    \centering
    \includegraphics[scale=0.60]{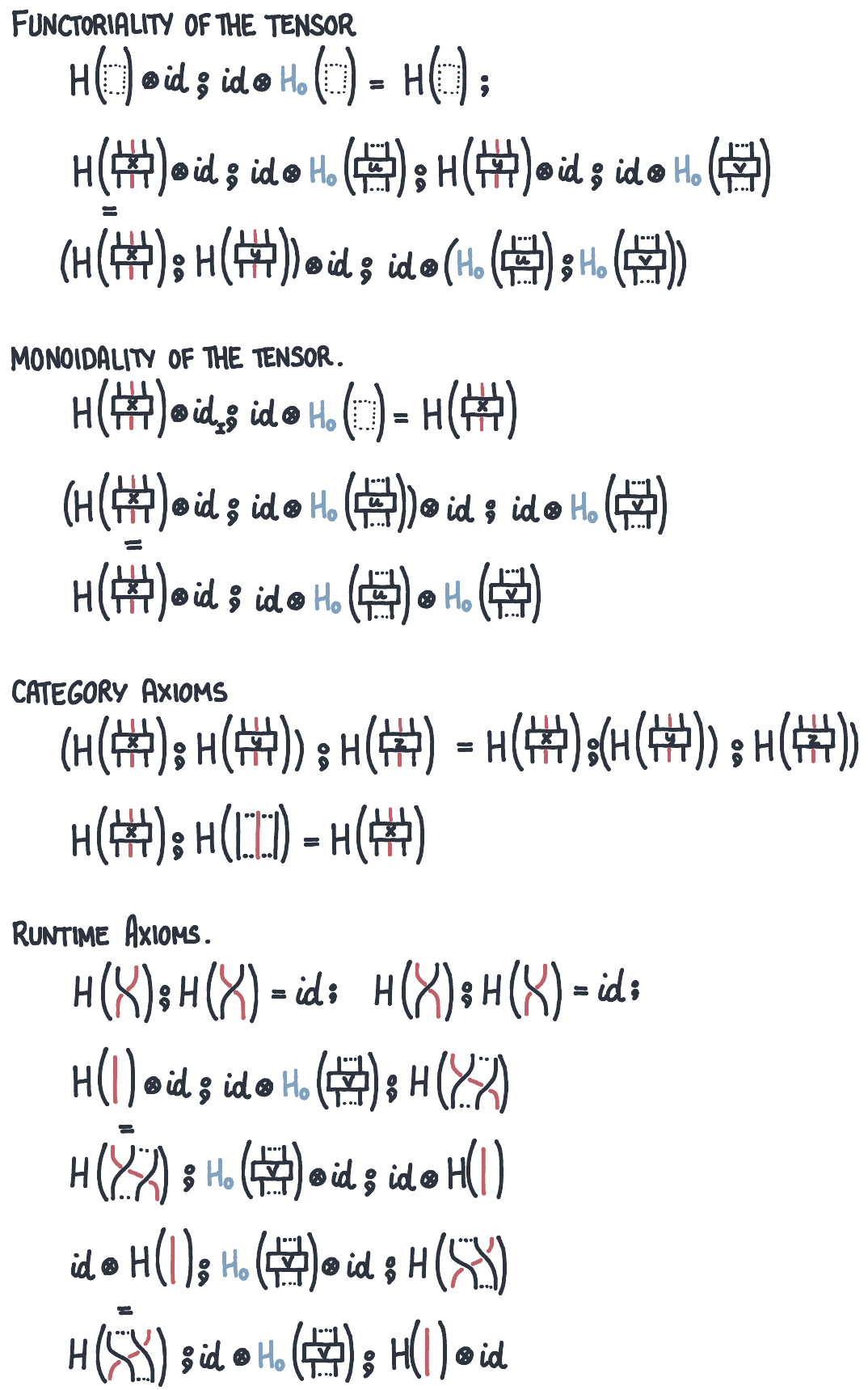}
    \caption{The assignment is well defined.}
    \label{fig:assignmentwell}
  \end{figure}

  \begin{itemize}
    \item
    The tensor is functorial. We know that $H(\im \tensor \im) = H(\im)$, both are identities and that can be formally proven by induction on the number of wires.
    Now, for the interchange law, consider a quartet of morphisms that can be composed or tensored first and such that, without loss of generality, we assume the runtime to be on the left side. Then, we can use centrality to argue that
    \begin{align*}
      H((x \tensor u) \comp (y \tensor v)) &= (H(x) \tensor \im) \comp (\im \tensor H_{0}(u)) \comp  (H(y) \tensor \im) \comp (\im \tensor H_{0}(v))
      \\&= ((H(x)\comp H(y)) \tensor \im) \comp (\im \tensor (H_{0}(u) \comp H_{0}(v)))
      \\&= H((x \comp y) \tensor (u \comp v)).
    \end{align*}

    \item
    The tensor is monoidal. We know that $H(x \tensor \im_{I}) = (H(x) \tensor \im_{I}) \comp (\im \tensor \im_{I}) = H(x)$.
    Now, for associativity, consider a triple of morphisms that can be tensored in two ways and such that, without loss of generality, we assume the runtime to be on the left side.
    Then, we can use centrality to argue that
    \begin{align*}
      H((x \tensor u) \tensor v) 
      &=  (((H(x) \tensor \im)  \comp (\im \tensor H_{0}(u))) \tensor \im) \comp \im \tensor H_{0}(v)
      \\&= (H(x) \tensor \im) \comp (\im \tensor H_{0}(u) \tensor H_{0}(v))
      \\&= H(x \tensor (u \tensor v))
    \end{align*}

    \item
    The terms form a category. And indeed, it is true by construction that $H(x \comp (y \comp z)) = H((x \comp y) \comp z)$ and also that $H(x \comp \im) = H(x)$ because $H$ preserves composition.

    \item
    The runtime category enforces some axioms. The composition of two braidings is mapped to the identity by the fact that $H$ preserves composition and sends both to the identity.  Both sides of the braid naturality over a morphism $v$ are mapped to $H_{0}(v)$; with the multiple braidings being mapped again to the identity.
  \end{itemize}
  Thus, $H$ is well-defined and it defines the only possible assignment and the only possible strict premonoidal functor.
\end{proof}

\section{Promonads}

\begin{lemma}[Kleisli category of a promonad]
  \label{ax:lemma:kleisli}
  Every promonad $(P,\starp,\unitp{})$ induces a category with the same objects as its base category, but with hom-sets given by $P(\bullet,\bullet)$, composition given by $(\starp)$ and identities given by $(\unitp{\im})$.
  This is called its \emph{Kleisli category}, $\kleisli{P}$. Moreover, there exists an identity-on-objects functor $\catC \to \kleisli{P}$, defined on morphisms by the unit of the promonad.
\end{lemma}
\begin{proof}
  Indeed, let us show that the composition defined by $(\starp)$ is unital and associative. Given any $p \in P(A,B)$, we have that identities are neutral with respect to composition on the right, $p \starp \unitp{\im}_{B} = p \rp \im_{B} = p$,
  and on the left $\unitp{\im}_{A} \starp u  = \im_{B} \lp u = u$.
  Composition is associative by the definition of promonad (\Cref{definition:promonad}, iii).

  Let us now check that the unit of the promonad $(\unitp{})$ defines an identity-on-objects functor, which is to say that the assignment on morphisms is functorial.
  By construction, $(\unitp{\im})$ is the identity on $\kleisli{P}$.
  Let us show now that the unit of the promonad also preserves composition,
  \[\begin{aligned}
    \unitp{f} \starp \unitp{g} =
    (\unitp{\im} \starp \unitp{f}) \starp \unitp{g} = \unitp{\im} < f < g =
    \unitp{\im} < (f \comp g)
    =
    \unitp{\im} \starp \unitp{(f \comp g)} =
    \unitp{(f \starp g)}.\qquad\qedhere
    \end{aligned}
  \]
\end{proof}

\begin{theorem}
  \label{ax:th:promonadidonobjs}
  Promonads over a category $\catC$ correspond to identity-on-objects functors from the category $\catC$.
  Given any identity-on-objects functor $i \colon \catC \to \catD$ there exists a unique promonad over $\catC$ having $\catD$ as its Kleisli category: the promonad given by the profunctor $\hom_{\catD}(i(\bullet),i(\bullet))$.
\end{theorem}
\begin{proof}
  Note that the hom-sets of a category $\hom_{\catD}(i(\bullet),i(\bullet))$ form a profunctor with actions
  \[(p \rp f) = p \comp i(f), \mbox{ and } (g \lp p) = i(g) \comp p.\]
  We define the unit to be the assignment of the functor on morphisms. That is, $\unitp{f} = i(f)$. We define multiplication of the promonad to be composition in $\catD$. Let us now check the axioms of a promonad: premultiplication $\unitp{f} \starp p = i(f) \comp p = f > p$, postmultiplication $p \starp \unitp{g} = p \comp i(g) = p < g$, dinaturality $p \starp (f \lp q) = p \comp i(f) \comp q = (p \rp f) \starp q$, and associativity $(p_{1} \starp p_{2}) \starp p_{3} = (p_{1} \comp p_{2} \comp p_{3}) = p_{1} \starp (p_{2} \starp p_{3})$.
  We can conclude (with \Cref{lemma:kleisli}) that promonads coincide with identity-on-objects functors.
\end{proof}

\begin{proposition}
  \label{ax:prop:homomorphismcommuting}
  A \promonadHomomorphism{} between two \promonads{} understood as identity-on-objects functors, $\baseV \to \catC$ and $\baseW \to \catD$, is equivalently a pair of functors $(F_{0},F)$ that commute strictly with the two identity-on-objects functors on objects $F_{0}(X) = F(X)$ and morphisms $\unitp{F_{0}(f)} = F(\unitp{f})$.
\end{proposition}
\begin{proof}
  Given a \promonadHomomorphism{} $(F_{0},F)$, we will construct the pair of functors.
  One of them is already $F_{0}$.
  For the second one, we observe that $F(p_{1} \starp p_{2}) = F(p_{1}) \starp F(p_{2})$, and $F(\im) = \unitp{F_{0}(\im)} = \im$, making $F$ itself into the morphism assignment of a functor.
  Moreover, this functor, with the same assignment on objects as $F_{0}$, makes the square commute.

  Given any strictly commutative square of functors with $\unitp{F_{0}(f)} = F(\unitp{f})$, we can see that, by functoriality, $F$ induces a natural transformation determining a promonad homomorphism.
  \[F(f \lp p \rp g) = F(\unitp{f} \starp p \starp \unitp{g}) = \unitp{F_{0}(f)} \starp F(p) \starp \unitp{F_{0}(g)} = F_{0}(f) \lp F(p) \rp F_{0}(g).\]
  Again by functoriality, and by the commutativity of the square, we can see that it must also satisfy the promonad homomorphism axioms.
\end{proof}

\begin{theorem}
  \label{ax:th:universalpure}
  Let $\catC \colon \baseV \profarrow \baseV$ and $\catD \colon \baseW \profarrow \baseW$ be two \promonads{} and let $\catC \ast \catD \colon \baseV \times \baseW \to \baseV \times \baseW$ be their pure tensor.
  There exist a pair of \promonadHomomorphisms{} $L \colon \catC \times \baseW \to \catC \ast \catD$ and $R \colon \baseV \times \catD \to \catC \ast \catD$. These are universal in the sense that, for every pair of promonad homomorphisms, $A \colon \catC \times \baseW \to \catE$ and $B \colon \baseV \times \catD \to \catE$, there exists a unique promonad homomorphism $(A \vee B) \colon \catC \ast \catD \to \catE$ that commutes strictly with them, $(A \vee B) \comp L = A$ and $(A \vee B) \comp R = B$.
\end{theorem}
\begin{proof}
  We start by constructing $L \colon \catC \times \baseW \to \catC \ast \catD$ and $R \colon \baseV \times \catD \to \catC \ast \catD$.
  These are defined by $L(f,w) = f_{\catC} \comp w_{\baseW} = w_{\baseW} \comp f_{\catC}$ and $R(g,v) = g_{\catD} \comp v_{\baseV} = v_{\baseV} \comp g_{\catD}$.
  Let us check that $L$ is promonad homomorphism, $R$ follows a similar reasoning.
  \begin{itemize}
    \item $L((v,w') \lp (f,w)) = L(v \lp f, w' \comp w) = v \comp f \comp w' \comp w = v \comp  w' \comp
    f \comp w = (v,w') \lp L(f,w)$,
    \item $L((f,w) \rp (v,w')) = L(f \rp v, w \comp w') = f \comp v \comp w \comp w' =
      f \comp w \comp v \comp w' = L(f,w) \rp (v,w')$,
    \item $L((f,w) \starp (f',w'))
      = L(f \starp f', w \comp w')
      = f \comp f' \comp w \comp w'
      = f \comp w \comp f' \comp w'
      = L(f,w) \starp L(f',w')$,
   \item $L(\unitp{v},w) = v\comp w = \unitp{L_{0}(v,w)}$.
  \end{itemize}

  We will now construct the promonad homomorphism as $(A \vee B)(f_{\catC}) = A(\im,f)$ and $(A \vee B)(g_{\catD}) = B(g,\im)$.
  This definition is forced by commutation with $l$ and $r$ and automatically defines the promonad homomorphism on all the generators of $\catC \ast \catD$.
  We can see it is well-defined, with the most interesting case being proving that it preserves the interchange of morphisms: indeed,
  $(A \vee B)(f_{\catC} \comp w_{\baseW}) = A(f,\im) \comp B(\im,w) = A(f,\im) \comp ((\im,w) \lp B(\im,\im)) = A(f,\im) \rp (\im,w) = A(f,w) = (A \vee B)(w_{\baseW} \comp f_{\catC})$, and the case with $(v,g)$ is analogous.
\end{proof}

\section{Cliques}

\begin{definition}[Clique]
  In a category $\catC$, a \emph{clique} $(X,\theta)$, is a family of objects, $X_{i}$, indexed by a set $i \in I$, and a family of isomorphisms $\theta_{i,j} \colon X_{i} \to X_{j}$ such that $\theta_{ij} \comp \theta_{jk} = \theta_{ik}$ and $\theta_{ii} = \im$.
\end{definition}

\begin{definition}[Clique morphism]
  A morphism between two cliques in the same category, $f \colon (X,\theta)\to (Y,\psi)$, is a family of morphisms $f_{ij} \colon X_{i} \to Y_{j}$ making every possible square commute, which means that $\theta_{ij} \comp f_{jl} = f_{ik} \comp \psi_{kl}$.
\end{definition}

\begin{proposition}
  A clique morphism  $f \colon (X,\theta)\to (Y,\psi)$ is completely determined by its value between any two indices, $f_{ij} \colon X_{i} \to Y_{j}$.
\end{proposition}
\begin{proof}
  By the definition, $f_{kl} = \theta_{ki} \comp f_{ij} \comp \psi_{lk}^{-1}$, where we use that the clique is made up of isomorphisms.
\end{proof}

\section{Pseudomonoids}

\begin{theorem}
  \label{ax:th:freydpseudomonoid}
  An \effectfulCategory{} (or monoidal Freyd category) is a pseudomonoid on the monoidal 2-category of promonads with promonad homomorphism, promonad transformations and the \pureTensor{} of promonads.
  A pseudomonoid homomorphism between \effectfulCategories{} is an \effectfulFunctor{}.
\end{theorem}

\begin{proof}[Proof sketch]
  Consider the data for a pseudomonoid $(\baseV \to \catC,\otimes,I,\alpha,\lambda,\rho)$ of promonads with the \pureTensor{}.
  The promonad $\baseV \to \catC$ gives us the pair of categories that will form the \effectfulCategory{}.
  We have a pair of promonad homomorphisms, $(\otimes) \colon (\baseV \times \baseV \to \catC \ast \catC) \to (\baseV \to \catC)$ and $I \colon (1 \to 1) \to (\baseV \to \catC)$ that make $\baseV$ into a monoidal category and $\catC$ into a premonoidal category with the same unit and tensor.

  Finally, the promonad transformations are natural transformations.
  This means that the associator and unitor 2-cells are natural transformations describing the associators and unitors of the monoidal category $\baseV$. These are the same associators and unitors of the premonoidal category $\catC$. Crucially, because they are in $\baseV$, they are central with respect to every morphism in $\catC$.
\end{proof}

\section{Background material}

\subsection{Monoidal categories}

This section of the appendix has been repurposed from a similar summary of monoidal categories \cite{monoidalstreams}, but it contains only standard material on monoidal categories that we choose to repeat to fix conventions.

\begin{definition}[\cite{maclane78}]
  A \defining{linkmonoidalcategory}{\textbf{monoidal category}},
  $(\catC, \otimes, I, \alpha, \lambda, \rho)$, is a category $\catC$
  equipped with a functor $\tensor \colon \catC \times \catC \to \catC$,
  a unit $\sI \in \catC$, and three natural isomorphisms: the associator $\alpha_{\sA,\sB,\sC} \colon (\sA \tensor \sB) \tensor \sC \cong \sA \tensor (\sB \tensor \sC)$, the left unitor $\lambda_{\sA} \colon \sI \tensor \sA \cong \sA$ and
  the right unitor $\rho_{\sA} \colon \sA \tensor \sI \cong \sA$;
  such that $\alpha_{\sA,\sI,\sB} ; (\im_{\sA} \tensor \lambda_{\sB}) = \rho_{\sA} \tensor \im_{\sB}$ and
  $(\alpha_{A,B,C} \tensor \im) ; \alpha_{A,B \tensor C, D} ; (\im_{A} \tensor \alpha_{B,C,D}) = \alpha_{A\tensor B,C,D} ; \alpha_{A,B,C \tensor D}$.  A monoidal category is \emph{strict} if $\alpha$, $\lambda$ and $\rho$ are identities.
\end{definition}

\begin{definition}[Monoidal functor, \cite{maclane78}]\defining{linkmonoidalfunctor}{}
  Let \[(\catC,\tensor,\sI,\alpha^{\catC},\lambda^{\catC},\rho^{\catC})\mbox{ and } (\catD,\boxtimes,\sJ,\alpha^{\catD},\lambda^{\catD},\rho^{\catD})\] be \hyperlink{linkmonoidalcategory}{monoidal categories}.
  A \defining{linkmonoidalfunctor}{\emph{monoidal functor}} (sometimes called \emph{strong monoidal functor}) is a triple
  $(F,\varepsilon,\mu)$ consisting of a functor $F \colon \catC \to \catD$ and two natural
  isomorphisms $\varepsilon \colon \sJ \cong F(\sI)$ and $\mu \colon F(\sA \tensor \sB) \cong F(\sA) \boxtimes F(\sB)$;
  such that
  \begin{itemize}
    \item the associators satisfy \(
        \alpha^{\catD}_{FA,FB,FC} ; (\im_{FA} \tensor \mu_{B,C}) ; \mu_{A,B \tensor C}
        = (\mu_{A,B} \tensor \im_{FC}) ; \mu_{A \tensor B,C} ; F(\alpha^{\catC}_{A,B,C})\),
    \item the left unitor satisfies \((\varepsilon \tensor \im_{FA}) ; \mu_{I,A} ; F(\lambda^{\catC}_{A}) = \lambda^{\catD}_{FA}\)
    \item the right unitor satisfies \((\im_{FA} \tensor \varepsilon) ; \mu_{A,I} ; F(\rho^{\catC}_{FA}) = \rho^{\catD}_{FA}\).
  \end{itemize}
  A monoidal functor is a \emph{monoidal equivalence} if it is moreover an equivalence of categories.  Two monoidal categories are monoidally equivalent if there exists a monoidal equivalence between them.
\end{definition}

During most of the paper, we omit all associators and unitors from monoidal categories, implicitly using the \emph{coherence theorem} for monoidal categories.

\begin{theorem}[Coherence theorem, \cite{maclane78}]%
  \label{theorem:coherence}
  Every monoidal category is monoidally equivalent to a strict monoidal category.
\end{theorem}

\begin{definition}[Symmetric monoidal category, \cite{maclane78}]
  A \defining{linksymmetricmonoidalcategory}{\emph{symmetric monoidal category}}
  $(\catC, \otimes, I, \alpha, \lambda, \rho, \sigma)$ is a monoidal category
  $(\catC, \otimes, I, \alpha, \lambda, \rho)$ equipped with a braiding
  $\sigma_{A,B} \colon A \otimes B \to B \otimes A$, which satisfies the hexagon
  equation
  \[\alpha_{A,B,C} ; \sigma_{A,B \tensor C} ; \alpha_{B,C,A} = (\sigma_{A,B} \tensor \im) ; \alpha_{B,A,C} ; (\im \tensor \sigma_{A,C})\]
  and additionally satisifes $\sigma_{A,B} ; \sigma_{B,A} = \im$.
\end{definition}

\begin{definition}[\cite{maclane78}]
  A \defining{linksymmetricmonoidalfunctor}{symmetric monoidal functor}
  between two \hyperlink{linksymmetricmonoidalcategory}{symmetric monoidal categories} $(\catC, \sigma^{\catC})$
  and $(\catD, \sigma^{\catD})$ is a monoidal functor $F \colon \catC \to \catD$ such that $\sigma^{\catD} ; \mu = \mu ; F(\sigma^{\catC})$.
\end{definition}

\subsection{Sesquifunctors}

\begin{definition}
  \defining{linksesquifunctor}
  Let $\catA_{1},\dots,\catA_{n}$ and $\catB$ be categories.
  A \emph{sesquifunctor} $T \colon \catA_{1}, \dots, \catA_{n} \to \catB$ \cite{street1996categorical} is an assignment on objects and morphisms that is independently functorial on each variable. That is, the sesquifunctor is given by a family of functors
  \[T_{i}(X_{1},\dots,\bullet_{i}\dots,X_{n}) \colon \catA_{i} \to \catB\mbox{ for } i = 1,\dots,n.\]
  These functors coincide on objects, in that $T(X_{1},\dots,X_{n})$ is uniquely determined independently of the $T_{i}$ we use to define it.
\end{definition}

\begin{remark}
  Sesquifunctors form a multicategory with ordinary composition, which preserves single-variable functoriality. Moreover, they form a 2-multicategory with transformations between them.
\end{remark}

\begin{proposition}
  The multicategory of sesquiprofunctors is representable by the funny tensor product of categories $(\funny) \colon \CAT \times \CAT \to \CAT$.
  The funny tensor product of two categories, $\catC \funny \catD$, is computed as the following pushout, where $\catC_{0}$ and $\catD_{0}$ are the discrete categories on the objects of $\catC_{0}$ and $\catD_{0}$.
  \begin{figure}[H]
    \centering
    \begin{tikzcd}
      \catC_0 \times \catD_{0} \rar \dar &
      \catC \times \catD_0 \dar \\
      \catC_0 \times \catD  \rar &
      \catC \funny \catD
    \end{tikzcd}
  \end{figure}
  Explicitly, objects of the funny tensor product are pairs of objects. Morphisms are either morphisms in $\catC$, in $\catD$, or formal compositions of both, as it happens with the coproduct of monoids. See Weber, \cite{weber13}.
\end{proposition}

\subsection{Double categories}

\begin{definition}[Monoids and promonoids]
  A \emph{monoid} in a double category is an arrow $T \colon \catA \to \catA$ together with cells $m \in \hom(M \otimes M;1,1;M)$ and $e \in \cell(1;1,1;M)$, called multiplication and unit, satisfying unitality and associativity.
  \begin{figure}[H]
    \centering
    \includegraphics[scale=0.6]{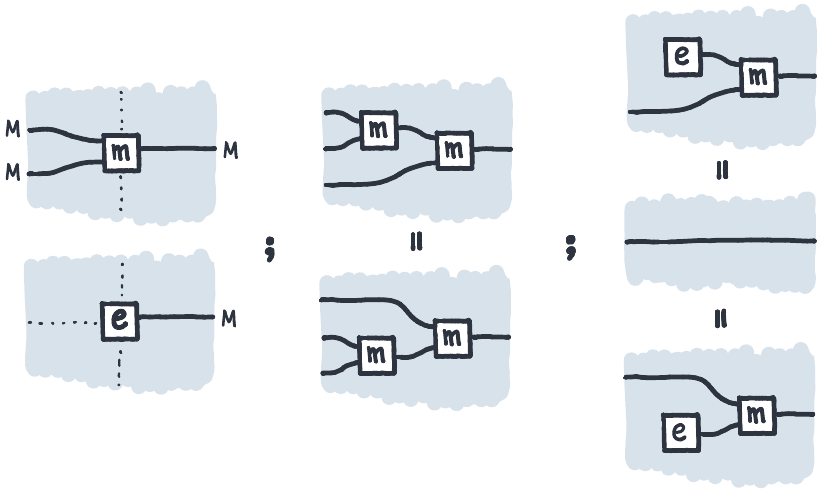}
    \caption{Data and axioms of a monoid in a double category.}
  \end{figure}
  A \emph{promonoid} in a double category is a proarrow $M \colon \catA \to \catA$ together with cells $m \in \cell(1;M \otimes M,M,1)$ and $e \in \cell(1;1,M;1)$, called promultiplication and prounit, satisfying unitality and associativity.
  \begin{figure}[H]
    \centering
    \includegraphics[scale=0.6]{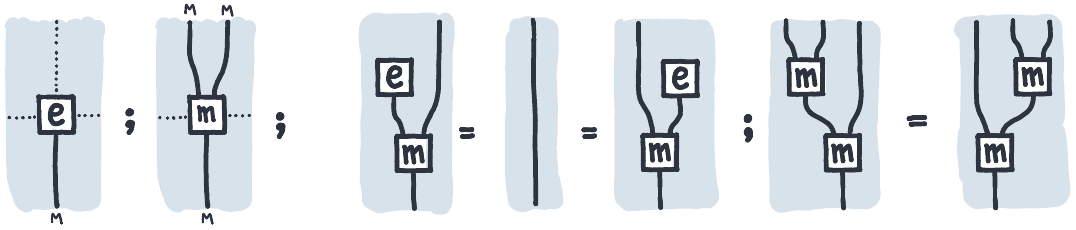}
    \caption{Data and axioms of a promonoid in a double category.}
  \end{figure}
  Dually, we can define \emph{comonoids} and \emph{procomonoids}.
\end{definition}

A monad is a monoid on the category of categories, functors and profunctors $\CAT$.

\subsection{Monoidal bicategories}

\begin{definition}\label{def:pseudomonoid}
  \defining{linkpseudomonoid}{}
  In a monoidal bicategory, a \emph{pseudomonoid} over a 0-cell $M$ is a pair of 1-cells, $M \boxtimes M \to M$ and $I \to M$, together with the following triple of invertible 2-cells representing associativity and unitality (\Cref{fig:pseudomonoid}), and satisfying the pentagon and triangle equations (\Cref{fig:pseudomonoidaxioms}).
  \begin{figure}[H]
    \centering
    \includegraphics[scale=0.6]{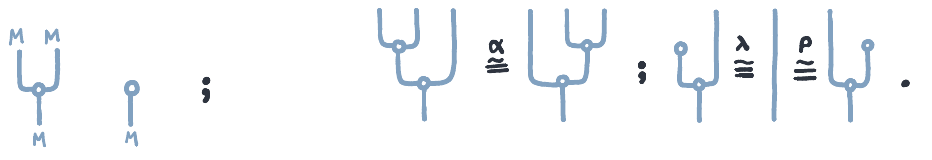}
    \caption{Data for a pseudomonoid.}
  \end{figure}
  \begin{figure}[H]
    \centering
    \includegraphics[scale=0.6]{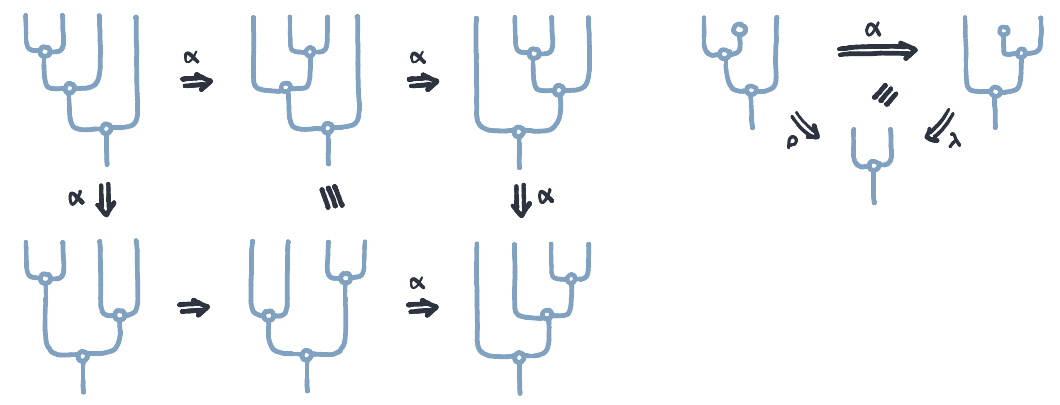}
    \caption{Pentagon and triangle axioms for a pseudomonoid.}
    \label{fig:pseudomonoidaxioms}
  \end{figure}
A \emph{symmetric} \pseudomonoid{} is a \pseudomonoid{} endowed with an invertible 2-cell representing commutation (\Cref{fig:commutator}), and satisfying symmetry and the two hexagon equations (\Cref{fig:hexagons}).
  \begin{figure}[H]
    \centering
    \includegraphics[scale=0.6]{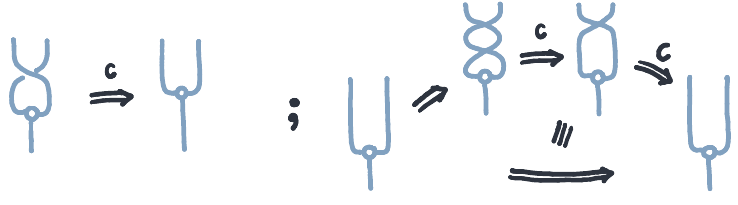}
    \caption{Commutator and symmetry for a pseudomonoid.}
    \label{fig:commutator}
  \end{figure}
  \begin{figure}[H]
    \centering
    \includegraphics[scale=0.6]{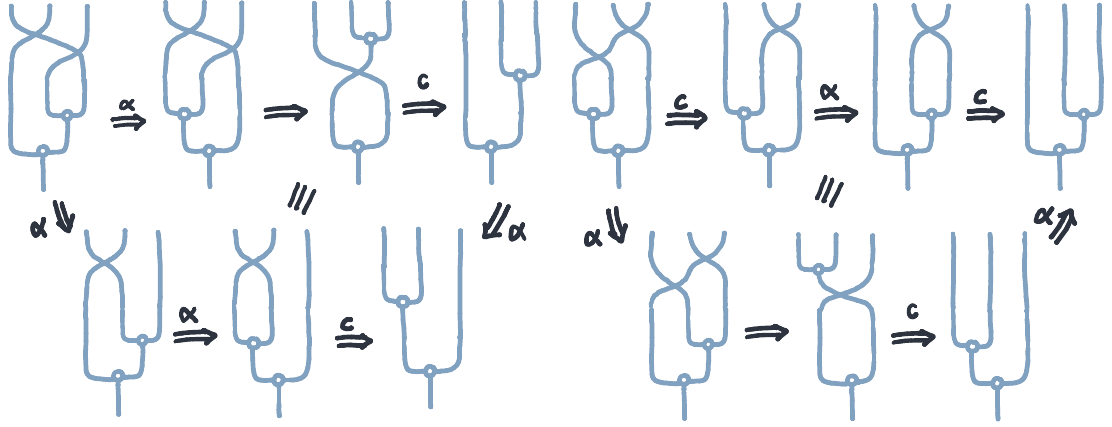}
    \caption{Hexagon equations for a symmetric pseudomonoid.}
    \label{fig:hexagons}
  \end{figure}
\end{definition}

\begin{definition}[Homomorphism of pseudomonoids]
  A homomorphism of pseudomonoids is given by a 1-cell between their underlying 0-cells and the following invertible 2-cells, representing preservation of the multiplication and the unit (\Cref{fig:pseudomonoidhom}), and satisfying compatibility with associativity and unitality (\Cref{fig:pseudomonoidhomaxioms}).

  \begin{figure}[H]
    \centering
    \includegraphics[scale=0.6]{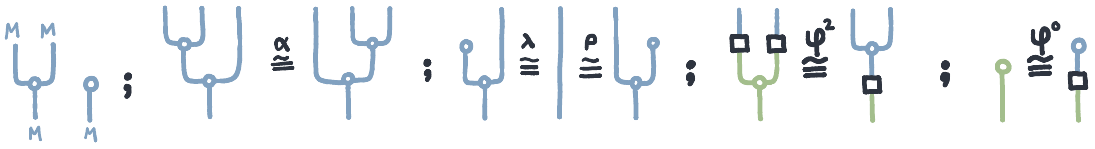}
    \caption{Data for a pseudomonoid homomorphism.}
  \end{figure}

  \begin{figure}[H]
    \centering
    \includegraphics[scale=0.6]{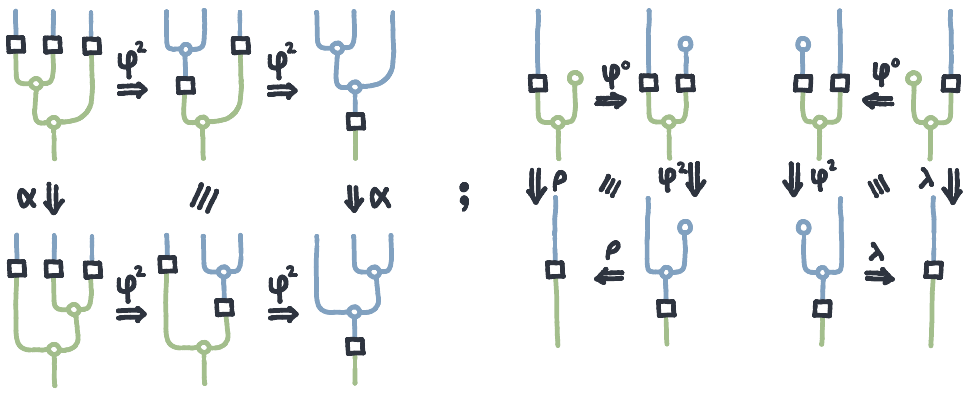}
    \caption{Axioms for a pseudomonoid homomorphism.}
    \label{fig:pseudomonoidhomaxioms}
  \end{figure}
\end{definition}

\end{document}